\documentclass[a4paper,11pt]{article}
\usepackage[margin=1in]{geometry}

\input{preamble.tex}

\newcommand{\TheAuthorADM}{Alberto De~Marchi}
\newcommand{\TheAffiliationADM}{Institute of Applied Mathematics and Scientific Computing, Department of Aerospace Engineering, University of the Bundeswehr Munich, Germany}
\newcommand{\TheAuthorTH}{Tim Hoheisel}
\newcommand{\TheAffiliationTH}{Department of Mathematics and Statistics, McGill University, Montréal, Canada}
\newcommand{\TheAuthorPM}{Patrick Mehlitz}
\newcommand{\TheAffiliationPM}{Department of Mathematics, Dresden University of Technology, Germany}
\newcommand{\TheTitle}{Augmented Lagrangian methods for fully convex composite optimization}

\newcommand{\TheKeywords}{%
	Convex composite optimization\keywordsAnd%
	Augmented Lagrangian framework\keywordsAnd%
	Proximal point algorithm\keywordsAnd%
	Convergence properties\keywordsAnd%
	Multiplier safeguarding%
}
\newcommand{\TheMSCClasses}{%
	\amsmscLink{49M37}\keywordsAnd%
	\amsmscLink{65K05}\keywordsAnd%
	\amsmscLink{90C25}\keywordsAnd%
	\amsmscLink{90C30}\keywordsAnd%
	\amsmscLink{90C46}
}%
\newcommand{\TheAbstract}{%
	This paper is concerned with augmented Lagrangian methods 
	for the treatment of fully convex composite optimization problems.
	We extend the classical relationship between augmented Lagrangian methods 
	and the proximal point algorithm to the inexact and safeguarded scheme
	in order to state global primal-dual convergence results.
	Our analysis distinguishes the regular case, where a stationary minimizer exists, 
	and the irregular case, where all minimizers are nonstationary.
	Furthermore, we suggest an elastic modification of the standard safeguarding scheme which preserves
	primal convergence properties while guaranteeing convergence of the dual sequence to a
	multiplier in the regular situation.
	Although important for nonconvex problems, 
	the standard safeguarding mechanism leads to weaker convergence guarantees for convex problems 
	than the classical augmented Lagrangian method.
	Our elastic safeguarding scheme combines the advantages of both while avoiding their shortcomings.
}

\begin{document}
	
	\title{\bfseries \TheTitle}
	\author{\TheAuthorADM\thanks{\TheAffiliationADM. \emailLink{alberto.demarchi@unibw.de}, \orcidLink{0000-0002-3545-6898}.}%
		\and%
		\TheAuthorTH\thanks{\TheAffiliationTH. \emailLink{tim.hoheisel@mcgill.ca}, \orcidLink{0000-0002-0782-6302}.}%
		\and%
		\TheAuthorPM\thanks{\TheAffiliationPM. \emailLink{patrick.mehlitz@tu-dresden.de}, \orcidLink{0000-0002-9355-850X}.}%
	}
	\date{}
	
	\maketitle
	
	\begin{center}
		\emph{Dedicated to Christian Kanzow on the occasion of his 60th birthday}
	\end{center}
	
	\begin{abstract}
		\TheAbstract
		\medskip
		
		\medskip
		
		\mscclasses{\TheMSCClasses}
	\end{abstract}
	
	\section{Introduction}\label{sec:intro}
	
	Consider a \emphdef{composite problem} of the form
	\begin{equation}
		\minimize_{x}\,
		\Phi(x) \coloneqq f(x) + g( c(x) ) ,
		\tag{P}\label{eq:P}
	\end{equation}
	where $\func{f}{\XX}{\R}$ is continuously differentiable, 
	$\func{g}{\YY}{\Rinf \coloneqq \R \cup \{+\infty\}}$ is lower semicontinuous, proper, convex, and 
	$\func{c}{\XX}{\YY}$ is continuously differentiable. 
	We assume throughout that \eqref{eq:P} is feasible, i.e., $\dom \Phi=c^{-1}(\dom g)$ is nonempty and, for algorithmic considerations, 
	that $g$ is prox-friendly, i.e., its proximal operator can be computed (at least inexactly) with reasonable effort. 
	Let us note that all results in this paper remain true if $\XX$ and $\YY$ are replaced by arbitrary finite-dimensional Hilbert spaces.
	Problem \eqref{eq:P} is, due to convexity of $g$, referred to as a \emphdef{convex composite problem} 
	or an \emphdef{extended nonlinear program} in the literature to which many people have contributed,
	most prominently, Burke et al.\ \cite{burke1985descent,burke1995gauss,burke2013epiconvergent,burke1992optimality}
	and Rockafellar \cite{rockafellar2000extended,rockafellar2023augmented,rockafellar2023convergence}.
	
	In this study, we focus on the case where $f$ and $g\circ c$, and hence $\Phi$, are convex, 
	i.e., we study \emphdef{fully convex composite problems}. 
	We refer the reader to \cref{ass:convexP} in \cref{sec:what_is_convex} for minimal assumptions that guarantee this. 
	Although this limits the range of the framework to a degree, the model problem \eqref{eq:P} covers diverse applications.
	The obvious ones are (regularized) linear least-squares and constrained convex programming (including semidefinite programming)
	where $f$ is the objective function, $g$ is an indicator function, and $c$ models the constraint functions.
	Some non-generic applications are matrix optimization problems involving variational Gram functions \cite{BurkeHoheiselGuan2019,JalaliFazelXiao2017},
	where $g$ is the support function of a compact set of positive semidefinite matrices 
	and $c(X)\coloneqq X X^\top$ for any input matrix $X$.
	The fully convex version of \eqref{eq:P} in which we are interested has been studied from a theoretical viewpoint by Bo\c{t} et al.\ \cite{Bot2010,bot2007constraint, bot2009generalized}
	and Burke et al.\ \cite{BurkeHoheiselGuan2019} who provide complete conjugacy and subdifferential results, as well as a plethora of applications.
	Doikov and Nesterov \cite{doikov2022highorder} discuss algorithmic approaches.
	
	In order to solve composite problems numerically, 
	the augmented Lagrangian method (ALM) turned out to deliver a suitable algorithmic framework,
	see, e.g., \cite{DhingraKhongJavanovic2019,HangSarabi2021,hang2023convergence,rockafellar2023convergence,rockafellar2024generalizations}
	for some recent contributions,
	as the emerging subproblems can be tackled via proximal-gradient methods,
	for example \cite{demarchi2022proximal,KanzowMehlitz2022}.
	
	Following the literature, ALMs enjoy convincing local convergence
	properties in the presence of suitable second-order conditions,
	see, e.g., \cite{HangSarabi2021,hang2023convergence,rockafellar2023convergence}
	for such results addressing convex composite optimization problems.
	However, in the convex composite situation, 
	global convergence properties, i.e., general assertions about primal
	accumulation points, are difficult to obtain.
	This has been illustrated in \cite{kanzow2017example}
	by means of a simple one-dimensional convex optimization problem
	whose feasible set is described via a nonconvex constraint.
	The gap of global convergence guarantees can be closed when relying on so-called \emphdef{safeguarding} techniques
	which keep the dual variables in the augmented Lagrangian scheme bounded,
	see, e.g., \cite{birgin2014practical} for a general view
	and the recent contributions \cite{demarchi2023constrained,DeMarchiMehlitz2024}
	for studies addressing general composite optimization problems.
	In numerical practice, safeguarded ALMs are typically preferred
	over the classical scheme due to their superior global properties.
	
	The convergence analysis of ALMs is closely related
	to the one of the famous proximal point algorithm (PPA) as the dual iterates typically
	turn out to correspond to the sequence generated by a (potentially inexact)
	PPA with variable stepsizes when applied to a suitable dual optimization problem.
	First results in this direction have been obtained 
	by Rockafellar in \cite{rockafellar1973dual,rockafellar1976augmented}
	and address smooth convex optimization problems with inequality constraints.
	However, an extension to convex composite problems is also possible, 
	see \cite{rockafellar2023convergence}
	where, however, this connection has not been used to infer
	global convergence properties of the ALM.
	In fact, the focus of \cite{rockafellar2023convergence} is on local properties,
	see \cite{rockafellar2023augmented} as well.
	Let us note that, in the safeguarded situation, the dual iterates correspond to the sequence
	generated by a generalized PPA, see \cite{kanzow2017generalized},
	which also does not exploit these observations to distill
	global convergence properties of the safeguarded ALM.
	The purpose of this paper is to complement the results 
	in \cite{kanzow2017generalized,rockafellar2023convergence}.
	
	\paragraph*{Contributions}
	
	Now we are in a position to outline the main contributions of this paper.
	To start,
	we \emph{analyze the fully convex composite model \eqref{eq:P} 
	and extend the well-known (dual) connection between the ALM and the PPA} to it.
	We give a generalized relationship between ALM and PPA 
	when applied to the broad problem class \eqref{eq:P} 
	capturing inexact solves of emerging subproblems.
	This relation is then used to infer \emph{two types of convergence results for the classical ALM}.
	First, we investigate the behavior of the sequence of dual iterates via the properties of PPA.
	Second, with the aid of these insights,
	it is shown that all accumulation points of the sequence of primal iterates 
	are global minimizers of \eqref{eq:P}.
	These results complement those in the recent papers \cite{andrews2025augmented,rockafellar2023convergence}
	in a straightforward way and provide the base for subsequent considerations.
	In \cite{andrews2025augmented}, the authors consider a more specific convex model problem \eqref{eq:P} 
	where the role of $g$ is played by an indicator function.
	In contrast, \eqref{eq:P} may be infeasible in \cite{andrews2025augmented} while we will assume throughout
	that \eqref{eq:P} possesses a (not necessarily stationary) minimizer.
	The focus of \cite{rockafellar2023convergence}, which is concerned with convex composite problems, 
	is on proving full convergence of the primal ALM iterates
	in the presence of so-called \emphdef{strong variational convexity} which amounts to a second-order condition in practice, 
	see, e.g., \cite{KhanhMordukhovichPhat2023}, which is not required in this paper.
	Particularly, we (separately) address the situation where all minimizers of \eqref{eq:P}
	are nonstationary, which is a setting that is not at all considered 
	in \cite{rockafellar2023convergence}.
	
	Second, we \emph{extend the PPA-based analysis to safeguarded ALMs},
	and to the best of our knowledge, our findings are largely new.
	Therefore, we relate the dual update of the method with
	a generalized PPA 
	which can be also interpreted as an instance 
	of the backward-backward splitting method from \cite{banert2014backward}
	whenever the safeguarding is implemented via projection onto the safeguarding set.
	This rather natural property allows us to distill finer convergence results
	than postulated in \cite{kanzow2017generalized}, 
	and we are in a position to incorporate inexact evaluations of the proximal operator.
	These findings allow us to infer primal-dual convergence results for the safeguarded ALM 
	which are similar to the ones obtained for the classical ALM.
	However, one has to be slightly more careful with the update of the penalty parameter 
	in the safeguarded setting.
	The fact that accumulation points of the primal sequence are global minimizers of \eqref{eq:P}
	whenever the update of the penalty parameter is somewhat reasonable 
	is well known and can be found,
	e.g., in \cite{birgin2014practical,DeMarchiMehlitz2024,KanzowKraemerMehlitzWachsmuthWerner2025}.
	Whenever the safeguarding set is too small, however,
	it may happen that the dual sequence does not capture Lagrange multipliers in case of existence,
	and this also affects the behavior of the sequence of primal iterates, see \cite{andreani2008augmented}.
	
	The third contribution of the paper
	overcomes this shortcoming:
	we suggest to adopt an \emph{elastic safeguarding},
	a slight modification of the standard technique  
	which allows for an enlargement of the safeguarding set in certain situations.
	With this mechanism we can infer that the sequence of dual iterates of the associated ALM
	always converges to a Lagrange multiplier in case of existence,
	while keeping the desirable global convergence properties of the primal sequence.
	Hence, this approach combines the advantages of both classical and (rigidly) safeguarded ALMs. 
	To the best of our knowledge, the latter findings are completely new,
	and we conjecture that, at least in parts, they can be extended to
	more general composite optimization problems.
	
	\paragraph*{Roadmap}
	
	After introducing some notation and recalling preliminary results in \cref{sec:preliminaries}, 
	we will present some basics about fully convex composite problems in \cref{sec:what_is_convex}.
	The connection between classical ALM and PPA is recalled in \cref{sec:classical_ALM}, 
	before investigating safeguarded ALMs in \cref{sec:safeguarded_ALM}.
	The proposed elastic safeguarding technique is
	discussed in \cref{sec:safeguarded_ALM_elastic}.
	\cref{sec:summary_and_numerics} presents some numerical experiments
	to illustrate our findings.
	Finally, some remarks and open questions are collected in \cref{sec:conclusions}.
	
	For the classical ALM from \cref{alg:ALMclassic}, the convergence guarantees in \cref{thm:global_convergence,thm:global_convergence_irregular} are obtained 
	exploiting the (relatively standard) ``equivalence'' between PPA and ALM iterates (\cref{thm:inexactPPA:saddle}) and the characterization of PPA (\cref{lem:PPA}).
	Let us mention that neither \cref{lem:PPA} nor \cref{thm:inexactPPA:saddle,thm:global_convergence} are new per se,
	as these findings merely complement results from the literature.
	However, we decided to include them to keep the paper self-contained.
	Furthermore, proofs are presented to set up the stage 
	for the considerations in \cref{sec:convergence_rigid_safeguard,sec:convergence_elastic_safeguard}
	where related techniques are used to obtain the main results of the paper.
	Indeed, for the ALM with rigid safeguard from \cref{alg:ALMsafeguarded},
	\cref{lem:PPA} is replaced by \cref{lem:dual_sequences_safeguarded}, 
	which inspects the ``generalized'' PPA \eqref{eq:generalized_PPA} arising from the safeguarding mechanism \cite{kanzow2017generalized}.
	Then, patterning the (dual) analysis for the classical ALM, 
	\cref{thm:global_convergence_safeguarded_alm_dual,thm:global_convergence_irregular_safeguarded} provide slightly weaker guarantees, 
	and so does \cref{cor:global_convergence_safeguarded_alm_primal}, 
	due to the arbitrary choice of the safeguarding set.
	Whenever the penalty parameter is updated according to a certain well-established rule, see, e.g., \cite{birgin2014practical},
	convergence guarantees can be strengthened, 
	see \cref{cor:global_convergence_safeguarded_alm_primal_HPR} and \cref{thm:global_convergence_irregular_safeguarded_primal}.
	Allowing the safeguard to be elastic,
	our analysis recovers convergence guarantees on par with classical ALM whenever stationary minimizers exist
	(\cref{thm:global_convergence_special_safeguarded_alm}),
	by overcoming the limitation of an arbitrarily prescribed safeguarding set.
	Since \cref{thm:global_convergence_special_safeguarded_alm} provides stronger guarantees than \cref{thm:global_convergence_safeguarded_alm_dual}, 
	the elastic safeguarding of \cref{alg:ALMsafeguarded_growing} should be preferred over the rigid one of \cref{alg:ALMsafeguarded}, 
	even in the regular case.
	An analogous verdict is valid for the irregular case, 
	with \cref{thm:global_convergence_special_safeguarded_alm_irregular} surpassing \cref{thm:global_convergence_irregular_safeguarded}.
	Interestingly, \cref{thm:global_convergence_special_safeguarded_alm_irregular} even outnumbers 
	the comparatively weak convergence guarantees for the classical ALM in the irregular setting (\cref{thm:global_convergence_irregular}),
	an achievement attributed to safeguarding.
	Hence, \cref{alg:ALMsafeguarded_growing} surpasses both \cref{alg:ALMclassic,alg:ALMsafeguarded} 
	in terms of primal-dual global convergence guarantees.
	\cref{tab:summary} presents an overview of the findings presented in the paper.
	
	\begin{table}[tbh]
		\centering%
		\begin{tabular}{cccc}
			\hline
			ALM & regular case & irregular case & \\
			\hline\hline
			classical & \cref{thm:global_convergence} & \cref{thm:global_convergence_irregular} & using \cref{lem:PPA,thm:inexactPPA:saddle} \\
			\hline
			& \cref{thm:global_convergence_safeguarded_alm_dual}
			& \multirow{2}{*}{\cref{thm:global_convergence_irregular_safeguarded}} & \multirow{2}{*}{using \cref{lem:dual_sequences_safeguarded}} \\
			rigid safeguard & \cref{cor:global_convergence_safeguarded_alm_primal} & &\\
			& \cref{cor:global_convergence_safeguarded_alm_primal_HPR} & \cref{thm:global_convergence_irregular_safeguarded_primal} & for \cref{set:algencan_penalty_parameter} \\
			\hline
			elastic safeguard & \cref{thm:global_convergence_special_safeguarded_alm} &  \cref{thm:global_convergence_special_safeguarded_alm_irregular} & \\
			\hline
		\end{tabular}%
		\caption{Summary of results on classical and safeguarded ALMs for \eqref{eq:P}.}%
		\label{tab:summary}%
	\end{table}
	
	Focusing at the situation where all minimizers are stationary,
	the attentive reader may wonder as to
	why to construct an ALM with elastic safeguard to just recover the guarantees of the classical ALM.
	In order to make this clear, 
	one has to observe that the latter scheme fails to have meaningful global convergence properties 
	(in the sense of arbitrary initial points) already 
	for general convex composite problems, see, e.g., \cite{kanzow2017example} again,
	whereas the former maintains this useful property 
	(proving this fact is simple, based on results for the rigid safeguard, such as \cite[Thm~6.2]{birgin2014practical}).
	Thus, the elastic safeguard enables good behavior in both the convex and nonconvex settings.
	The latter is important as, in numerical practice, it might be an involved or even impossible task
	to check convexity of some arbitrary given optimization problem.

	\section{Notation and preliminaries}
	\label{sec:preliminaries}
	
	This section introduces notational conventions, preliminary results on PPA, 
	the Lagrangian framework in convex composite optimization, and other useful tools.
	
	\subsection{Fundamentals}
	\label{sec:notation}
	
	For notation, we mainly follow the classical book \cite{rockafellar1998variational}. 
	
	We denote the real numbers by $\R$ and set $\Rinf\coloneqq\R\cup\{+\infty\}$ .
	We let $\R_+$ and $\R_-$ be the set of nonnegative and nonpositive real numbers, respectively.
	The set of positive integers is denoted by $\N$, and we set $\N_0 \coloneqq \N \cup \{0\}$.
	We equip the appearing Euclidean spaces with inner product $\innprod{\cdot}{\cdot}$ with the associated Euclidean norm $\|\cdot\|$.
	We will make use of the so-called \emphdef{four-point inequality}, see, e.g., \cite[Eq.\ (3)]{banert2014backward},
	given by
	\begin{equation}\label{eq:four_point_inequality}
		\forall x,y,z,w\in\YY\colon\quad
		\norm{x-y}^2 + \norm{z-w}^2
		\leq
		\norm{x-z}^2 + \norm{x-w}^2 + \norm{y-z}^2 + \norm{y-w}^2
		.
	\end{equation}
	
	Given a set $C\subseteq \YY$ and a point $\bar z\in\YY$, 
	we use $\bar z + C \coloneqq C+\bar z\coloneqq \{z+\bar z\in\YY \,|\, z\in C\}$ for brevity.
	The notation $\{z^k\}_{k\in K}$ represents a sequence indexed by elements of the set $K\subseteq \N$, 
	and we write $\{z^k\}_{k\in K}\subseteq C$ to indicate that $z^k \in C$ for all indices $k\in K$.
	Whenever clear from context, we may simply write $\{z^k\}$ to indicate $\{z^k\}_{k\in\N}$. 
	Notation $z^k\to_K \bar z$ is used to express convergence of $\{z^k\}_{k\in K}$ to $\bar z$.
	If $n=1$, we use $\{t_k\}$ to emphasize that we are dealing with sequences of scalars
	(using subscripts instead of superscripts), 
	and subsequences are denoted similarly as mentioned above.
	Given a scalar $t\in\R$ and a sequence $\{t_k\}\subseteq\R$, the notation $t_k \to t$ means that $t_k$ converges to $t$; 
	convergence from above is denoted $t_k \downto t$; $t_k \downtoneq t$ means that additionally $t_k \ne t$ for all $k\in\N$.
	
	Below, we present a helpful result which can be used to verify full convergence of sequences.
	It is a finite-dimensional counterpart of \cite[Lem.\ 5]{banert2014backward}.
	As we slightly weakened the assumptions, a proof is included for the purpose of completeness.
	\begin{mybox}
		\begin{lemma}\label{lem:Banerts_lemma}
			Let $\{y^k\}\subseteq\YY$ be a sequence, and let $Y\subseteq\YY$ be a nonempty set.
			Assume that each accumulation point of $\{y^k\}$ belongs to $Y$
			and that, for each $y\in Y$, $\{\norm{y^k-y}\}$ converges.
			Then there exists $\bar y\in Y$ such that $y^k\to\bar y$.
		\end{lemma}
	\end{mybox}
	\begin{proof}
		As $Y$ is nonempty, we find $y\in Y$,
		and the convergence of $\{\norm{y^k-y}\}$
		immediately yields the boundedness of $\{y^k\}$ via the inverse triangle inequality.
		Hence, $\{y^k\}$ possesses accumulation points.
		We now prove that $\{y^k\}$ possesses a unique accumulation point $\bar y$
		which then already must be the limit of $\{y^k\}$,
		and $\bar y\in Y$ follows from the assumptions.
		
		Suppose that $\hat y,\check y\in Y$ are accumulation points of $\{y^k\}$.
		Then, for each $k\in\N$,
		we have $\norm{\hat y-\check y}\leq\norm{\hat y-y^k}+\norm{y^k-\check y}$.
		By assumption, the sequences $\{\norm{\hat y-y^k}\}$ and $\{\norm{y^k-\check y}\}$
		converge, and as $\hat y$ and $\check y$ are accumulation points of $\{y^k\}$,
		their limit must be $0$.
		Hence, $\norm{\hat y-y^k}+\norm{y^k-\check y}\to 0$ follows,
		which yields $\hat y=\check y$.
	\end{proof}
	
	For a given nonempty, closed, convex set $C\subseteq \YY$, 
	the \emphdef{distance} $\func{\dist_C}{\YY}{[0,\infty)}$ of $C$ is $\dist_C(y)\coloneqq \inf_{z\in C} \norm{z-y}$,
	and the uniquely determined point $z_y\in C$ which satisfies $\dist_C(y)=\norm{z_y-y}$
	is denoted by $\proj_C(y)\coloneqq z_y$.
	Mapping $\func{\proj_C}{\YY}{\YY}$ is called the \emphdef{projection} of $C$.
	With $\func{\indicator_C}{\YY}{\Rinf}$ we denote the \emphdef{indicator} of $C$, 
	given by $\indicator_C(z)=0$ if $z\in C$ and $\indicator_C(z)=\infty$ otherwise.
	We refer to $C^\circ\coloneqq\{y\in\YY\,|\,\forall z\in C\colon\,\innprod{y}{z}\leq 0\}$ as the \emphdef{polar cone} of $C$
	and note that it is a nonempty, closed, convex cone.
	
	The \emphdef{domain} of a function $\func{\varphi}{\YY}{\R\cup\{\pm\infty\}}$ is denoted by $\dom \varphi \coloneqq \{ z\in\YY \,|\, \varphi(z)<\infty \}$.
	We say that $\varphi$ is \emphdef{proper} if $\dom \varphi$ is a nonempty set on which $\varphi$ is finite.
	Moreover, $\varphi$ is \emphdef{lower semicontinuous} (lsc) if $\varphi(\bar{z}) \leq \liminf_{z\to\bar{z}} \varphi(z)$ is valid for all $\bar{z}\in\YY$.
	The set $\epi \varphi \coloneqq \{ (z,\alpha) \in \YY \times \R \,|\, \alpha\geq \varphi(z) \}$ is called the \emphdef{epigraph} of $\varphi$.
	Note that $\varphi$ is lsc if and only if $\epi \varphi$ is closed.
	Whenever $\varphi$ is convex and $\bar z\in\dom \varphi$ is a point where $\varphi$ is finite,
	\[
	\partial \varphi(\bar z)\coloneqq \{y\in\YY\,|\,\forall z\in\dom \varphi\colon\,\varphi(z)\geq \varphi(\bar z) + \innprod{y}{z-\bar z}\}
	\]
	is called the \emphdef{subdifferential} of $\varphi$ at $\bar z$.
	Partial subdifferentiation will be denoted in canonical way using appropriate subscripts.
	For $\varphi\coloneqq\indicator_C$, where $C\subseteq\YY$ is nonempty, closed, and convex,
	and $\bar z\in C$, $N_C(\bar z)\coloneqq\partial\indicator_C(\bar z)$
	is the \emphdef{normal cone} to $C$ at $\bar z$.
	
	We remind also that the \emphdef{conjugate} of a proper, lsc, convex function $\func{\varphi}{\YY}{\Rinf}$ 
	is the proper, lsc, convex function $\func{\varphi^\conj}{\YY}{\Rinf}$ defined as $\varphi^\conj(y) \coloneqq \sup_{z\in\YY} \{ \innprod{z}{y}-\varphi(z)\}$, 
	and that one then has $y \in\partial \varphi(z)$ if and only if $z \in\partial \varphi^\conj(y)$.
	The \emphdef{support} $\func{\support_C}{\YY}{\Rinf}$ of a nonempty, closed, convex set $C\subseteq \YY$ is the convex function
	given by $\support_C(y)\coloneqq \indicator^\conj_C(y)= \sup_{z\in C} \innprod{y}{z}$.
	
	Let $\func{\varphi}{\YY}{\Rinf}$ be proper, lsc, and convex.
	For $\gamma>0$, the proper, lsc, convex function $\func{\varphi^\gamma}{\YY}{\R}$ given by
	\[
	\varphi^\gamma(z)
	\coloneqq
	\inf_w\left\{\varphi(w)+\frac1{2\gamma}\norm{w-z}^2\right\}
	\]
	is referred to as the \emphdef{Moreau envelope} of $\varphi$ for $\gamma$.
	Observe that $\varphi^\gamma$ is continuously differentiable, 
	see \cite[Prop.\ 12.29]{BauschkeCombettes2011}.
	Noting that $\varphi + \nicefrac12\norm{\cdot - z}^2$ is uniformly convex for each $z\in\YY$,
	there exists a uniquely determined point $w_z\in\YY$ such that
	$\varphi^1(z)=\varphi(w_z)+\nicefrac12\norm{w_z-z}^2$,
	and we set $\prox_\varphi(z)\coloneqq w_z$.
	The associated mapping $\func{\prox_{\varphi}}{\YY}{\YY}$ 
	is called the \emphdef{proximal mapping} of $\varphi$.
	Note that $\varphi^\gamma$ is the marginal function associated with $\prox_{\gamma\varphi}$ for each $\gamma>0$.
	We will exploit the relation
	\begin{equation}\label{eq:three_point_inequality}
		\forall z,w\in\YY\colon\quad
		\varphi(\prox_\varphi(z))
		\leq
		\varphi(w) + \frac{1}{2} \norm{z-w}^2 - \frac{1}{2} \norm{z-\prox_\varphi(z)}^2 - \frac{1}{2} \norm{w-\prox_\varphi(z)}^2,
	\end{equation}
	which can be found in \cite[Eq.\ (4)]{banert2014backward}.
	
	Given a function $\func{\Theta}{\XX}{\YY}$ and $y\in\YY$, $\func{\innprod{y}{\Theta}}{\XX}{\R}$ defined via
	$\innprod{y}{\Theta}(x)\coloneqq\innprod{y}{\Theta(x)}$ is the \emphdef{scalarization function} of $\Theta$ given $y$.
	Whenever $\Theta$ is differentiable at $\bar x\in\XX$, $\Theta'(\bar x)\in\R^{m\times n}$ denotes its \emphdef{Jacobian}.
	In the case $m\coloneqq 1$, $\nabla \Theta(\bar x)\coloneqq \Theta'(\bar x)^\top$ represents the \emphdef{gradient} of $\Theta$.
	Accordingly, partial derivatives are denoted in canonical way.  
	
	\subsection{Recession cones and functions}
	
	For a nonempty, closed, convex set $C\subseteq\YY$, we use
	\[
	C_\infty \coloneqq \{v\in\YY\,|\,\forall z\in C\colon\,z+v\in C\}
	\]
	to denote the \emphdef{recession cone} of $C$ which is a nonempty, closed, convex cone,
	see, e.g., \cite[Thms~8.1 and~8.2]{Rockafellar1970}.
	Furthermore, for a proper, lsc, convex function $\func{\varphi}{\YY}{\Rinf}$, 
	we refer to the proper, lsc, convex function $\func{\varphi_\infty}{\YY}{\Rinf}$
	given via
	$
	\epi\varphi_\infty \coloneqq (\epi\varphi)_\infty
	$
	as the \emphdef{recession function} of $\varphi$. 
	
	The following lemma shows that a convex function and its Moreau envelope share the same recession function.
	Its mainly technical proof is presented in \hyperref[sec:appendix]{Appendix~\ref{sec:appendix}}.
	\begin{mybox}
		\begin{lemma}\label{lem:recession_function_of_moreau_envelope}
			Let $\func{\varphi}{\YY}{\Rinf}$ be proper, lsc, and convex.
			Then, for each $\gamma>0$, we have $(\varphi^\gamma)_\infty=\varphi_\infty$, i.e.,
			the recession functions of $\varphi$ and its Moreau envelope $\varphi^\gamma$ are the same.
		\end{lemma}
	\end{mybox}

	In the course of the paper, for a proper, lsc, convex function $\func{\varphi}{\YY}{\Rinf}$,
	we make use of
	\[
	\hzn \varphi \coloneqq \{v\in\YY\,|\,\forall z\in\dom\varphi\colon\,\varphi(z+v)\leq\varphi(z)\},
	\]
	the so-called \emphdef{horizon} of $\varphi$. 
	In the subsequently stated lemma, we summarize some properties of the horizon operation.
	
	\begin{mybox}
		\begin{lemma}\label{lem:horizon}
			Let $\func{\varphi}{\YY}{\Rinf}$ be a proper, lsc, convex function.
			Then the following assertions hold.
			\begin{enumerate}[label=(\roman{*})]
				\item\label{item:horizon_explicit} 
				We have $\hzn\varphi = \{v\in\XX\,|\,\varphi_\infty(v)\leq 0\}=(\dom \varphi^\conj)^\circ$.
				Particularly, $\hzn\varphi$ is a nonempty, closed, convex cone.
				\item\label{item:horizon_monotonicity} 
				For $z_1,z_2\in\XX$ such that $z_1-z_2\in\hzn \varphi$, 
				$\varphi(z_1)\leq\varphi(z_2)$ holds true.
				\item\label{item:horizon_of_moreau_envelope}
				For each $\gamma>0$, $\hzn \varphi^\gamma = \hzn \varphi$.
			\end{enumerate}
		\end{lemma}
	\end{mybox}
	\begin{proof}
		The equations in assertion~\ref{item:horizon_explicit} 
		are due to \cite[Thm~8.7]{Rockafellar1970} and \cite[Thm~14.2]{Rockafellar1970}, respectively.
		The final statements follow from the properties of the polar cone.
		
		For the proof of assertion~\ref{item:horizon_monotonicity}, pick $z_1,z_2\in\YY$ 
		such that $z_1-z_2\in\hzn \varphi$. This yields $\varphi(z+(z_1-z_2))\leq\varphi(z)$
		for each $z\in\dom\varphi$.
		If $z_2\notin\dom\varphi$, $\varphi(z_1)\leq\varphi(z_2)$ is trivially satisfied.
		Otherwise, we can pick $z\coloneqq z_2$ to obtain this relation as well.
		
		Finally, let us prove assertion~\ref{item:horizon_of_moreau_envelope}.
		Picking $\gamma>0$ arbitrarily, we can apply \cref{lem:recession_function_of_moreau_envelope}
		and assertion~\ref{item:horizon_explicit} in order to find
		\[
		\hzn\varphi
		=
		\{v\in\YY\,|\,\varphi_\infty(v)\leq 0\}
		=
		\{v\in\YY\,|\,(\varphi^\gamma)_\infty(v)\leq 0\}
		=
		\hzn\varphi^\gamma 
		\]
		which already yields the claim.
	\end{proof}
	
	Let us note that assertion~\ref{item:horizon_monotonicity} of \cref{lem:horizon}
	can already be found in \cite[Lem.\ 7]{BurkeHoheiselNguyen2021}
	where a more involved proof is presented.

	\subsection{On inexact PPA}\label{sec:foundations_of_PPA}
	
	We consider a given proper, lsc, and convex function $\func{\varphi}{\YY}{\Rinf}$.
	The exact proximal operator $\prox_{\gamma\varphi}$ of $\varphi$ with stepsize $\gamma>0$ 
	has been introduced in \cref{sec:notation}.
	Let us represent the \emphdef{inexact proximal mapping} of $\varphi$ with stepsize $\gamma>0$ by
	\begin{equation}\label{eq:inexact_prox_mapping}
		\prox_{\gamma\varphi}^\varepsilon(z)
		\coloneqq
		\left\{
		y \in \YY
		\,\middle\vert\,
		\norm{y-\prox_{\gamma\varphi}(z)}\leq\sqrt{2\gamma\varepsilon}
		\right\}
	\end{equation}
	for each $z\in\YY$ and $\varepsilon\geq 0$.
	A similar notion of inexactness for proximal operators is exploited in 
	\cite{andrews2025augmented,rockafellar1976monotone}.
	It is apparent that, for all $z\in\YY$, $\prox_{\gamma\varphi}^0(z) = \{\prox_{\gamma\varphi}(z)\}$ 
	and, for all $\varepsilon > 0$, 
	$\prox_{\gamma\varphi}(z) \in \prox_{\gamma\varphi}^\varepsilon(z)$.
	Notice that, whenever $\varepsilon>0$, $\prox_{\gamma\varphi}^\varepsilon$ is set-valued.
	Furthermore, this notion of inexactness is closely related to proximal objective suboptimality.
	Indeed, picking $y\in\YY$ such that
	\[
	\varphi(y) + \frac{1}{2\gamma} \norm{y-z}^2
	\leq
	\varphi^\gamma(z) + \varepsilon,
	\]
	the $\frac{1}{\gamma}$-uniform convexity of the function $\Phi_z\coloneqq\varphi+\nicefrac{1}{2\gamma}\norm{\cdot-z}^2$
	and the inclusion $0\in\partial \Phi_z(\prox_{\gamma\varphi}(z))$ for the exact proximal evaluation lead to the inequalities
	\begin{equation*}
		\varepsilon \geq \Phi_z(y) - \Phi_z(\prox_{\gamma\varphi}(z)) \geq \frac{1}{2\gamma} \norm{y-\prox_{\gamma\varphi}(z)}^2,
	\end{equation*}
	i.e., $y\in\prox_{\gamma\varphi}^\varepsilon(z)$.
	Hence, the notion of inexactness exploited here is weaker, and thus allows for more flexibility,
	than the one based on proximal objective suboptimality,
	which was adopted, among others, in \cite{salzo2012inexact}.
	
	The (inexact) PPA we are dealing with is presented in \cref{alg:PPA}.
	It allows for adaptive stepsizes (bounded away from zero) and exploits the inexact proximal operator
	defined in \eqref{eq:inexact_prox_mapping}.
	
	\begin{algorithm2e}[htb]
		\DontPrintSemicolon
		\KwData{$\underline\gamma>0$; $y^0\in\YY$\;}
		\For{$k = 0,1,2\ldots$}{
			Select $\gamma_k\geq\underline\gamma$ and $\varepsilon_k\geq 0$.
			\;
			Compute $y^{k+1}\in\prox_{\gamma_k\varphi}^{\varepsilon_k}(y^k)$.
		}
		\caption{Inexact PPA for minimization of $\varphi$.}
		\label{alg:PPA}
	\end{algorithm2e}
	
	In the following proposition, we present convergence properties of \cref{alg:PPA}.
	We note that these results can also be retrieved 
	from the classical findings in \cite[Thm~1]{rockafellar1976monotone}.
	Regardless, we present here an explicit and self-contained proof,
	inspired by the one of \cite[Thm~1]{banert2014backward},
	so we can fall back on several of its arguments when we prove some of our main results in 
	\cref{lem:dual_sequences_safeguarded}
	as well as
	\cref{thm:global_convergence_safeguarded_alm_dual,thm:global_convergence_special_safeguarded_alm}.

	\begin{mybox}
		\begin{proposition}\label{lem:PPA}
			Let $\func{\varphi}{\YY}{\Rinf}$ be proper, lsc, convex, and bounded below.
			Let $\{y^k\}$ be a sequence generated by \cref{alg:PPA}
			such that $\sum_{k=0}^\infty\sqrt{\gamma_k\varepsilon_k}<\infty$.
			Then the following hold:
			\begin{enumerate}[label=(\roman{*})]
				\item\label{item:PPA_summable}
				$\sum_{k=0}^{\infty}\gamma_k^{-1}\norm{y^{k+1}-y^k}^2<\infty$, hence $\gamma_k^{-1}(y^{k+1}-y^k)\to 0$;		
				\item\label{item:PPA_diverges} 
				if $\varphi$ does not possess a minimizer,
				then $\norm{y^k}\to\infty$;
				\item\label{item:PPA_converges} 
				if $\varphi$ possesses a minimizer, 
				then $\{y^k\}$ converges to a minimizer of $\varphi$.
			\end{enumerate}
		\end{proposition}
	\end{mybox}
	\begin{proof}
		To start, let us prove assertion \ref{item:PPA_summable}.
		Note that we have
		\[
		\sum_{k=0}^\infty\sqrt{\varepsilon_k}
		=
		\frac{1}{\sqrt{\underline{\gamma}}}\sum_{k=0}^\infty\sqrt{\underline{\gamma}\varepsilon_k}
		\leq
		\frac{1}{\sqrt{\underline{\gamma}}}\sum_{k=0}^\infty\sqrt{\gamma_k\varepsilon_k}
		<
		\infty,
		\]
		and, thus, as necessarily $\varepsilon_k\downto 0$, we have
		$\sum_{k=0}^\infty\varepsilon_k<\infty$.
		Similarly, we find that $\sum_{k=0}^\infty\gamma_k\varepsilon_k<\infty$.
		For each $k\in\N_0$, we define the exact proximal step as
		$
		y^{k+1}_{\rm e}\coloneqq \prox_{\gamma_k\varphi}(y^k)
		$.
		Then we have
		\begin{equation}\label{eq:inexact_prox_bound_distance_concrete}
			\norm{y^{k+1}-y^{k+1}_{\rm e}}\leq\sqrt{2\gamma_k\varepsilon_k}
		\end{equation}
		for each $k\in\N_0$.
		For each $k\in\N$, \eqref{eq:three_point_inequality} gives us
		\[
		\varphi(y^{k+1}_{\rm e})
		\leq
		\varphi(y^k_{\rm e}) 
		+ 
		\frac{1}{2\gamma_k}\left( \norm{y^k-y^k_{\rm e}}^2 - \norm{y^{k+1}_{\rm e}-y^k_{\rm e}}^2 \right).
		\]
		Summing up the above inequalities for $k=1,\ldots,N$ yields
		\begin{equation}\label{eq:inexact_prox_bound_telescope}
			\varphi(y^{N+1}_{\rm e})
			\leq
			\varphi(y^1_{\rm e})
			+
			\frac12\sum_{k=1}^N\frac{1}{\gamma_k}
			\left( \norm{y^k-y^k_{\rm e}}^2 - \norm{y^{k+1}_{\rm e}-y^k_{\rm e}}^2 \right).
		\end{equation}
		Hence, we find
		\begin{align*}
			\sum_{k=1}^N\frac{1}{\gamma_k}\norm{y^{k+1}-y^k}^2
			\leq{}&
			\sum_{k=1}^N\frac{1}{\gamma_k}\left(
			\norm{y^{k+1}-y^{k+1}_{\rm e}}
			+
			\norm{y^{k+1}_{\rm e}-y^k_{\rm e}}
			+
			\norm{y^k_{\rm e}-y^k}
			\right)^2
			\\
			\leq{}&
			3\sum_{k=1}^N\frac{1}{\gamma_k}\left(
			\norm{y^{k+1}-y^{k+1}_{\rm e}}^2
			+
			\norm{y^{k+1}_{\rm e}-y^k_{\rm e}}^2
			+
			\norm{y^k_{\rm e}-y^k}^2
			\right)
			\\
			\leq{}&
			6\sum_{k=0}^{\infty}\varepsilon_k
			+
			3\sum_{k=1}^N\frac{1}{\gamma_k}\norm{y^{k+1}_{\rm e}-y^k_{\rm e}}^2
			+
			\frac{6}{\underline{\gamma}}\sum_{k=0}^\infty\gamma_k\varepsilon_k
			\\
			\leq{}&
			6\sum_{k=0}^{\infty}\varepsilon_k
			+
			\frac{12}{\underline{\gamma}}\sum_{k=0}^\infty\gamma_k\varepsilon_k
			+
			6\left( \varphi(y^1_{\rm e})-\inf\varphi \right),
		\end{align*}
		where the last two inequalities are due to \eqref{eq:inexact_prox_bound_distance_concrete}
		and \eqref{eq:inexact_prox_bound_telescope}.	
		Noting that the right-hand side of this estimate is a finite constant,
		taking the limit $N\to\infty$ yields the first statement in \ref{item:PPA_summable}.
		The latter implies $\gamma_k^{-1}\norm{y^{k+1}-y^k}^2\to 0$,
		and noting that $\{\gamma_k^{-1}\}$ is bounded,
		$\gamma_k^{-1}(y^{k+1}-y^k)\to 0$ follows.
		
		As a next step, we aim to show that each accumulation point of $\{y^k\}$
		is a minimizer of $\varphi$.
		Therefore, pick a subsequence $\{y^{k+1}\}_{k\in K}$ and some point $\tilde y\in\YY$
		such that $y^{k+1}\to_K\tilde y$.
		Then \eqref{eq:inexact_prox_bound_distance_concrete} immediately yields $y^{k+1}_{\rm e}\to_K\tilde y$.
		By definition of the proximal mapping,
		we find $0\in\partial(\varphi + \nicefrac{1}{2\gamma_k}\norm{\cdot-y^k}^2)(y^{k+1}_{\rm e})$
		or, equivalently,
		\begin{equation}\label{eq:OC_from_prox}
			-\frac{1}{\gamma_k}(y^{k+1}_{\rm e}-y^k) 
			\in 
			\partial\varphi(y^{k+1}_{\rm e})
		\end{equation}
		for each $k\in K$.
		Noting that
		\[
		\frac{1}{\gamma_k}(y^{k+1}_{\rm e}-y^k)
		=
		\frac{1}{\gamma_k}(y^{k+1}_{\rm e}-y^{k+1})
		+
		\frac{1}{\gamma_k}(y^{k+1}-y^k)
		\to
		0
		\]
		holds according to the boundedness of $\{\gamma_k^{-1}\}$, 
		\eqref{eq:inexact_prox_bound_distance_concrete}, and the first part of the proof,
		taking the limit $k\to_K\infty$ in \eqref{eq:OC_from_prox} while respecting 
		\cite[Thm~24.4]{Rockafellar1970} yields $0\in\partial\varphi(\tilde y)$.
		This, however, means that $\tilde y$ is a minimizer of $\varphi$.
		
		The above immediately proves assertion~\ref{item:PPA_diverges}.
		Indeed, if $\{y^k\}$ possesses a bounded subsequence, then the latter possesses an accumulation
		point which would be a minimizer of $\varphi$. This, however, is a contradiction.
		
		To complete the proof, let us verify assertion~\ref{item:PPA_converges}.
		Therefore, let $\bar y\in\YY$ be a minimizer of $\varphi$.
		Noting that $\bar y$ is a fixed point of $\prox_{\gamma\varphi}$ for each $\gamma>0$,
		by \eqref{eq:inexact_prox_bound_distance_concrete} we find
		\begin{multline*}
			\norm{y^{k+1}-\bar y}
			\leq
			\norm{y^{k+1}-y^{k+1}_{\rm e}} + \norm{y^{k+1}_{\rm e} - \bar y}
			\\
			\leq
			\sqrt{2\gamma_k\varepsilon_k} + \norm{\prox_{\gamma_k\varphi}(y^k)-\prox_{\gamma_k\varphi}(\bar y)}
			\leq
			\sqrt{2\gamma_k\varepsilon_k} + \norm{y^k-\bar y}
		\end{multline*}
		for each $k\in\N_0$,
		exploiting
		the nonexpansiveness of the proximal operator.
		The above yields that the sequence $\{\norm{y^k-\bar y}-\sum_{\ell=0}^{k-1}\sqrt{2\gamma_\ell\varepsilon_\ell}\}$
		is monotonically nonincreasing, since
		\begin{equation*}
			\norm{y^{k+1}-\bar y}-\sum_{\ell=0}^{k}\sqrt{2\gamma_\ell\varepsilon_\ell}
			\leq
			\sqrt{2\gamma_k\varepsilon_k} + \norm{y^k-\bar y}-\sum_{\ell=0}^{k}\sqrt{2\gamma_\ell\varepsilon_\ell} 
			=
			\norm{y^k-\bar y}-\sum_{\ell=0}^{k-1}\sqrt{2\gamma_\ell\varepsilon_\ell}
		\end{equation*}
		holds for all $k\in\N$,
		and bounded below by the finite constant $-\sum_{\ell=0}^\infty\sqrt{2\gamma_k\varepsilon_k}$.
		Hence, it converges, 
		and the calculus rules for convergent sequences yield that $\{\norm{y^k-\bar y}\}$ converges.
		Noting that the argumentation is valid for each minimizer of $\varphi$,
		the assertion follows from \cref{lem:Banerts_lemma}.
	\end{proof}
	
	\subsection{Lagrangian framework}
	
	Let us compile here some underlying concepts of ALMs.
	Introducing a (Lagrange) multiplier $y\in\YY$, we define 
	the \emphdef{Lagrangian} function $\func{\LL}{\XX\times\YY}{\R\cup\{-\infty\}}$ of \eqref{eq:P} by
	\begin{equation*}
		\LL(x,y)
		\coloneqq
		f(x) + \inf_z \left\{ g(z) + \innprod{y}{c(x)-z} \right\}
		=
		f(x) + \innprod{y}{c(x)} - g^\conj(y)
		.
	\end{equation*}
	By definition, the mapping $\LL(x,\cdot)$ is concave for each $x\in\XX$.
	Rockafellar \cite{rockafellar2023convergence} refers to $\LL$ as the \emphdef{generalized Lagrangian} function
	associated to \eqref{eq:P}.
	For more details, the interested reader may also study the discussion 
	in \cite[Rem.\ 3.5]{DeMarchiMehlitz2024}.
	By \cite[Thm~11.1]{rockafellar1998variational} we have
	\begin{equation}\label{eq:saddleLL}
		\Phi(x)
		=
		f(x)+g(c(x))
		=
		f(x)+g^{\conj\conj}(c(x))
		=
		f(x)+\sup_{y}\left\{ \innprod{y}{c(x)}-g^\conj(y) \right\}
		=
		\sup_y \LL(x,y)
	\end{equation}
	for each $x\in\XX$.
	Following the same pattern,
	adding a quadratic penalty for the constraint violation
	leads to the definition of the \emphdef{augmented Lagrangian} function
	$\func{\LL_\mu}{\XX \times \YY}{\R}$ associated to \eqref{eq:P}, given by
	\begin{align}\label{eq:ALMz:AugLagrangian}
		\LL_\mu(x,y)
		{}\coloneqq{}&
		f(x) + \inf_z \left\{ g(z) + \innprod{y}{c(x)-z} + \frac{1}{2 \mu} \norm{c(x) - z}^2 \right\} \\
		={}&
		f(x)
		+ \inf_z \left\{ g(z) + \frac{1}{2 \mu} \norm{c(x) + \mu y - z}^2 \right\} - \frac{\mu}{2} \norm{y}^2 \nonumber\\
		{}={}&
		f(x) + g^\mu( c(x) + \mu y ) - \frac{\mu}{2} \norm{y}^2
		\nonumber
	\end{align}
	for some penalty parameter $\mu>0$.
	Let us note that, 
	according to \cite[Prop.\ 13.24\,(iii)]{BauschkeCombettes2011} and \cite[Thm~11.1]{rockafellar1998variational},
	we have
	\begin{align*}
		\LL_\mu(x,y)
		&=
		f(x) + (g^\mu)^{\conj\conj}(c(x)+\mu y) - \frac{\mu}{2}\norm{y}^2
		\\
		&=
		f(x) + \sup_z\left\{\innprod{z}{c(x)+\mu y} - (g^\mu)^{\conj}(z)\right\} - \frac{\mu}{2}\norm{y}^2
		\\
		&=
		f(x) + \sup_z\left\{\innprod{z}{c(x)+\mu y} - g^\conj(z) - \frac{\mu}{2}\norm{z}^2 - \frac{\mu}{2}\norm{y}^2\right\}
		\\
		&=
		f(x) + \sup_z\left\{\innprod{z}{c(x)} - g^\conj(z) - \frac{\mu}{2}\norm{z-y}^2\right\}
		\\
		&=
		\sup_z\left\{\LL(x,z) - \frac{\mu}{2}\norm{z-y}^2\right\}.
		\numberthis\label{eq:rep_AL_func}
	\end{align*}
	
	The underlying idea of the ALM is now to
	minimize, in each iteration, the augmented Lagrangian $\LL_\mu(\cdot,y)$,
	and to perform a suitable update of the penalty parameter as well as the multiplier afterwards.
	Note that our assumptions will guarantee that $\LL_\mu(\cdot,y)$ is convex which
	allows for its (approximate) global minimization, see \cref{lem:convex_setting} below.
	
	Let us comment on some more elementary terminology regarding \eqref{eq:P}.
	A feasible point $\bar x\in\XX$ of \eqref{eq:P} is called \emphdef{stationary} 
	whenever there exists $\bar y\in\YY$ 
	such that $0 \in \partial_x \LL(\bar x,\bar y)$ and $0 \in \partial_y (-\LL)(\bar x,\bar y)$,
	i.e., conditions
	\begin{equation*}
		\nabla_x\LL(\bar x,\bar y)=0
		,\qquad
		c(\bar x) \in \partial g^\conj(\bar y)
	\end{equation*}
	are valid.
	Owing to convexity of $g$, the latter system is equivalent to
	\begin{equation*}
		\nabla_x\LL(\bar x,\bar y)=0
		,\qquad
		\bar y\in\partial g(c(\bar x))
		,
	\end{equation*}
	and, thus, implicitly requires feasibility of $\bar x$.
	Let us note that stationarity of $\bar x$ implies $0\in\partial\Phi(\bar x)$,
	and thus, as $\Phi$ is assumed to be convex, minimality of $\bar x$ for $\Phi$.
	The converse is true in the presence of a qualification condition.
	Exemplary, the condition
	\begin{equation}\label{eq:QC}
		c'(\bar x)^\top y=0
		,\,
		y\in N_{\dom g}(c(\bar x))
		\quad\implies\quad
		y=0
	\end{equation}
	can be used for that purpose, see, e.g., \cite[Prop.\ 5.2]{BurkeHoheisel2017},
	as it implies $\partial(g\circ c)(\bar x)=c'(\bar x)^\top\partial g(c(\bar x))$.
	Throughout the paper, however, we will sometimes merely assume the existence of a
	minimizer of \eqref{eq:P} which is stationary,
	and disregard qualification conditions.
	Finally, let
	\[
	\optiYY(\bar x)\coloneqq\{y\in\partial g(c(\bar x))\,|\,\nabla_x\LL(\bar x,y)=0\}
	\]
	be the set of multipliers associated with $\bar x$.
	The set-valued mapping $\ffunc{\optiYY}{\XX}{\YY}$ is
	the so-called multiplier mapping and
	possesses a closed graph, 
	see \cite[Ex.\ 13.30]{rockafellar1998variational}.
	
	\section{Fully convex composite optimization}
	\label{sec:what_is_convex}
	
	Throughout the remainder of the paper, 
	let us assume that \eqref{eq:P} satisfies the following assumptions
	which, in parts, have already been outlined in \cref{sec:intro}.
	\begin{mybox}
		\begin{assum}\label{ass:convexP}
			The following hold for \eqref{eq:P}:
			\begin{itemize}
				\item 
				$f$ is continuously differentiable and convex;
				\item 
				$g$ is proper, lsc, convex,
				and its proximal mapping is numerically accessible;	
				\item 
				$c$ is continuously differentiable and $(-\hzn g)$-convex, i.e.,
				for all $x,x'\in\XX$ and $\lambda\in[0,1]$, we have
				\[
				c(\lambda x+(1-\lambda)x') - \lambda c(x) - (1-\lambda)c(x') \in \hzn g;
				\]
				\item 
				problem \eqref{eq:P} is feasible, possesses a minimizer,
				and its minimal value is $\Phi^\star$.
			\end{itemize}
		\end{assum}
	\end{mybox}
	
	Below, we summarize some essential consequences of this setting.
	Among others, it is shown that \cref{ass:convexP} guarantees that
	\eqref{eq:P} is a convex problem, and that the augmented Lagrangian function
	associated with \eqref{eq:P}, see \eqref{eq:ALMz:AugLagrangian},
	is convex in $x$ for each $\mu>0$ and $y\in\YY$.
	
	\begin{mybox}
		\begin{lemma}\label{lem:convex_setting}
			Suppose that \cref{ass:convexP} holds.
			Then the following hold.
			\begin{enumerate}[label=(\roman{*})]
				\item\label{item:convex_composition}
				The function $g\circ c$ is convex
				and, particularly, $\Phi$ is convex.
				\item\label{item:convex_scalarization}
				For each $y\in \dom g^\conj$, the scalarization function $\innprod{y}{c}$ is convex.
				\item \label{item:convex_composition_envelope}
				For each $\mu>0$ and for each $y\in\YY$, the mapping $g^\mu(c(\cdot)+\mu y)$ is convex.
				Particularly, $\LL_\mu(\cdot,y)$ is convex.
				\item\label{item:convex_AL_function_concave_in_y}
				For each $\mu>0$ and for each $x\in\XX$, $\LL_\mu(x,\cdot)$ is concave.
			\end{enumerate}
		\end{lemma}
	\end{mybox}
	\begin{proof}
		Assertion~\ref{item:convex_composition} is an immediate consequence of \cref{lem:horizon}\,\ref{item:horizon_monotonicity}
		and the $(-\hzn g)$-convexity of $c$.
		
		In order to prove \ref{item:convex_scalarization}, we emphasize that
		$\dom g^\conj\subset(\hzn g)^\circ=-(-\hzn g)^\circ$ holds due to \cref{lem:horizon}\,\ref{item:horizon_explicit},
		so that the assertion immediately follows from \cite[Thm~3a]{BurkeHoheiselNguyen2021}.
		
		For the proof of assertion~\ref{item:convex_composition_envelope}, we first observe that the mapping
		$c(\cdot)+\mu y$ is $(-\hzn g)$-convex as well since the same holds for $c$.
		Furthermore, \cref{lem:horizon}\,\ref{item:horizon_of_moreau_envelope} yields that it is already
		$(-\hzn g^\mu)$-convex as $\hzn g^\mu$ and $\hzn g$ coincide.
		Now, convexity of $g^\mu(c(\cdot)+\mu y)$ follows as in the proof of assertion~\ref{item:convex_composition}.
		Finally, $\LL_\mu(\cdot,y)$ is convex as a sum of convex functions.
		
		Assertion~\ref{item:convex_AL_function_concave_in_y} readily follows from representation \eqref{eq:rep_AL_func},
		keeping in mind that $\LL(x,\cdot)$ is concave.
	\end{proof}
	
	\begin{mybox}
		\begin{remark}\label{rem:convex_setting_via_cone}
			Let us note that, in order to obtain convexity of the composition $g\circ c$, one could also consider the more general
			setting where, given a nonempty, closed, convex cone $K\subseteq\YY$, the following properties hold, 
			see \cite[Prop.\ 1]{BurkeHoheiselNguyen2021} as well:
			\begin{enumerate}[label=(\roman{*})]
				\item\label{item:-K_convexity} 
				$c$ is $(-K)$-convex, i.e., for all $x,x'\in \XX$ and $\lambda\in[0,1]$, we have
				\[
				c(\lambda x+(1-\lambda)x')-\lambda c(x)-(1-\lambda)c(x')\in K,
				\]
				\item\label{item:-K_increasing} 
				$g$ is convex and $(-K)$-increasing, i.e., for all $z_1,z_2\in\YY$, we have
				\begin{equation*}
					z_1-z_2\in K\quad\implies \quad g(z_1)\leq g(z_2).
				\end{equation*}
			\end{enumerate}
			One can observe, however, that whenever property~\ref{item:-K_increasing} is valid for a given closed, convex cone $K\subseteq\YY$,
			then $K\subseteq\hzn g$. Indeed, for $v\in K$ and $z\in\dom g$, due to $(z+v)-z\in K$, the characterization in \ref{item:-K_increasing}
			gives $g(z+v)\leq g(z)$, and $v\in\hzn g$ follows.
			Hence, in order to achieve a maximum amount of freedom in the choice of $c$ satisfying~\ref{item:-K_convexity},
			picking $K\coloneqq \hzn g$ is expedient.
			Related discussions can be found in \cite[\S 5]{GisslerHoheisel2023}.
			It remains open whether the property in \cref{lem:convex_setting}\,\ref{item:convex_composition_envelope} can be
			verified whenever $K$ is chosen different from $\hzn g$.
		\end{remark}
	\end{mybox}
	
	\begin{mybox}
		\begin{example}\label{ex:recover_setting_of_KanzowSteck}
			Let us consider the special case where, for a nonempty, closed, convex set $C\subseteq\YY$,
			$g\coloneqq\indicator_C$ is chosen.
			Noting that $g^\conj=\support_C$ holds true, we find $\dom g^\conj = (C_\infty)^\circ$, 
			and \cref{lem:horizon}\,\ref{item:horizon_explicit} yields $\hzn g = C_\infty$ since
			$C_\infty$ is a closed, convex cone.
			Hence, in this situation, our standing assumption on $c$ boils down to demanding that
			$c$ is $(-C_\infty)$-convex,
			and that is precisely the setting which has been investigated, 
			e.g., in \cite{steck2018dissertation}.
		\end{example}
	\end{mybox}
	
	\subsection{On the value function of a fully convex composite problem}
	
	Next, we will study properties of the \emphdef{value function} 
	$\func{\mathcal{V}}{\YY}{\R\cup\{\pm\infty\}}$ of \eqref{eq:P}, 
	defined by
	\begin{equation}\label{eq:value_function}
		\mathcal{V}(u)
		\coloneqq
		\inf_x \overline{\Phi}(u,x),
	\end{equation}
	see \cite{rockafellar1974conjugate},
	where $\func{\overline{\Phi}}{\YY\times\XX}{\Rinf}$
	is the composite mapping given by
	\begin{equation}\label{eq:composite_mapping_perturbed}
		\overline{\Phi}(u,x)\coloneqq f(x) + g(c(x) - u).
	\end{equation}
	Owing to \cref{ass:convexP},
	$\overline{\Phi}$ is a proper, lsc, and (jointly) convex function.
	
	\begin{mybox}
		\begin{lemma}\label{lem:value_function}
			Let \cref{ass:convexP} hold, and let $\mathcal{V}$ be given by \eqref{eq:value_function}.
			Then the following hold.
			\begin{enumerate}[label=(\roman*)]
				\item\label{lem:value_function:cvx}
				Function $\mathcal{V}$ is convex. 
				\item\label{lem:value_function:proper}
				Function $\mathcal{V}$ is proper and lsc
				if there exists $\bar y\in \dom g^\conj$ such that one of the following (equivalent) conditions is satisfied: 
				\begin{enumerate}[label=(\alph{*})]
					\item\label{lem:value_function:proper1}
					$(-\dom{f^\conj})\cap \dom\innprod{\bar y}{c}^\conj\neq \emptyset$;
					\item\label{lem:value_function:proper2}
					$ f+ \innprod{\bar y}{c}$ is bounded below.
				\end{enumerate}
				In particular, this is the case if there exists $\bar y\in \dom g^\conj$ such that 
				the convex function $f+\innprod{\bar y}{c}$ attains its minimum,
				which is true under any of the following conditions: 
				\begin{enumerate}[label=(\alph{*}),resume]
					\item\label{lem:value_function:proper3}
					$f_\infty(x)+\innprod{\bar y}{c}_\infty(x)>0$ for all $x\in\XX\setminus\{0\}$;
					\item\label{lem:value_function:proper4}
					there exists a minimizer $\bar x\in\XX$ of \eqref{eq:P} which is stationary. 
				\end{enumerate}
			\end{enumerate}
		\end{lemma}
	\end{mybox}
	\begin{proof}
		We want to leverage \cite[Thm~3.101]{hoheisel2019topics}.
		To this end, we use $\overline{\Phi}$ defined in \eqref{eq:composite_mapping_perturbed} and
		note that \eqref{eq:value_function} holds for all $u\in\YY$. 
		Since $\overline{\Phi}$ is proper, lsc, and convex, \cite[Thm~3.101a)]{hoheisel2019topics} already yields part \ref{lem:value_function:cvx}.
		Then \cite[Thm~3.101d)]{hoheisel2019topics} stipulates that
		$\mathcal V$ is proper and lsc if $\dom \overline{\Phi}^\conj(\cdot,0)\neq \emptyset$.
		To compute this conjugate, we want to invoke \cite[Cor.\ 3]{BurkeHoheiselNguyen2021}.
		To this end, we define 
		$\func{\bar{f}}{\YY\times \XX}{\R}$ via $\bar{f}(u,x) \coloneqq f(x)$
		and  
		$\func{\bar{c}}{\YY\times \XX}{\YY}$ via $\bar{c}(u,x) \coloneqq c(x)-u$
		so that 
		\begin{equation}\label{eq:perturbed_problem_composite_setting}
			\overline{\Phi}(u,x)=\bar{f}(u,x)+g(\bar{c}(u,x)).
		\end{equation}
		We first observe that, by the assumptions on $f$, $g$, and $c$,
		the qualification condition in \cite[Eq.\ (17)]{BurkeHoheiselNguyen2021} is trivially satisfied
		(as $\bar{f}$ and $\bar{c}$ have full domain and the image of $\bar{c}$ is the whole space).
		Therefore, as we have $\dom g^\conj \subseteq (\hzn g)^\circ$ from \cref{lem:horizon}\,\ref{item:horizon_explicit},
		\cite[Cor.\ 3]{BurkeHoheiselNguyen2021} yields
		\begin{equation}\label{eq:PsiConj}
			\overline{\Phi}^\conj(p,q)
			=
			\inf_{y,r,s} \left\{g^\conj(y) + \bar{f}^\conj(r,s) + \innprod{y}{\bar{c}}^\conj\left((p-r,q-s)\right)\right\}.
		\end{equation}
		We now notice that
		\[
		\bar{f}^\conj(r,s)=\indicator_{\{0\}}(r) + f^\conj(s).
		\]
		Moreover, a short computation shows that 
		\[
		\innprod{y}{\bar{c}}^\conj(r,s) = \indicator_{\{0\}}(r+y) + \innprod{y}{c}^\conj(s).
		\]
		Inserting in \eqref{eq:PsiConj} now gives
		\begin{align*}
			\overline{\Phi}^\conj(p,0)
			={}&
			\inf_{y,r,s} \left\{g^\conj(y)+\indicator_{\{0\}}(r)+f^\conj(s)+\indicator_{\{0\}}(y+p-r)+ \innprod{y}{c}^\conj(-s)\right\}
			\\
			={}&
			g^\conj(-p)+\inf_{s} \left\{ f^\conj(s)+\innprod{-p}{c}^\conj(-s)\right\}.
		\end{align*}
		This already explains condition \ref{lem:value_function:proper1}.
		Condition \ref{lem:value_function:proper2} is equivalent to~\ref{lem:value_function:proper1}
		as, for $\bar y\in\dom g^\conj$, we have
		\[
		\inf_{s} \left\{ f^\conj(s) + \innprod{\bar y}{c}^\conj(-s)\right\}
		=
		(f+\innprod{\bar y}{c})^\conj(0)
		=
		-\inf \left\{f+ \innprod{\bar y}{c}\right\},
		\]
		where \cite[Thm~16.4]{Rockafellar1970} has been used for the the first identity.
		We observe that \cite[Thm~3.26, Cor.\ 3.27, and Ex.\ 3.29]{rockafellar1998variational} yield 
		that condition \ref{lem:value_function:proper3} implies \ref{lem:value_function:proper2}.
		
		Finally, let $\bar x\in\XX$ be a stationary minimizer of \eqref{eq:P} with associated multiplier $\bar y$.
		We first note that $\bar y\in \dom g^\conj$ (as $\dom \partial g^\conj\subseteq \dom g^\conj$).
		Moreover, 
		$\nabla f(\bar x)+c'(\bar x)^\top\bar y=0$
		implies that $\bar x$ is a minimizer of the smooth convex function $f+\innprod{\bar y}{c}$,
		see \cref{lem:convex_setting}\,\ref{item:convex_scalarization}.
		Hence, condition \ref{lem:value_function:proper4} implies \ref{lem:value_function:proper2}.
	\end{proof}
	
	The following example shows how the sufficient conditions for properness 
	and lower semicontinuity of $\mathcal{V}$ 
	from \cref{lem:value_function} can be verified.  
	
	\begin{mybox}
		\begin{example}
			Consider \eqref{eq:P} with 
			\[
			f(x)
			\coloneqq
			\log\left( \sum_{j=1}^n \exp(x_{j}) \right)
			,\quad
			g(y)
			\coloneqq
			\max_{i=1,\dots, m} \{y_i\}
			,\quad
			c(x)
			\coloneqq
			\left(\frac{1}{2}x^\top Q_ix\right)_{i=1}^m, 
			\]
			where $Q_i\in\R^{n\times n}$ is symmetric and positive semidefinite for all $i=1,\dots,m$.
			Then $\dom g^\conj=\Delta_m$ (see \cite[Ex.\ 4.4.11]{Beck2017}),
			where $\Delta_m\subseteq\YY_+$ is the standard simplex in $\YY$,
			$f_\infty(x) = \max_{j=1,\ldots,n} \{x_j\}$
			(see \cite[Eq.\ 3(5)]{rockafellar1998variational}), and
			for each $y\in \Delta_m$, we have 
			$
			\innprod{y}{c}_\infty
			=
			\sum_{i=1}^m (y_i c_i)_\infty
			=
			\sum_{i=1}^m \indicator_{\ker Q_i}
			=
			\indicator_{\cap_{i=1}^m \ker Q_i},
			$ 
			where \cite[Ex.\ 3.29]{rockafellar1998variational} has been used. 
			Altogether, we find 
			\[
			f_\infty(x)+\innprod{y}{c}_\infty(x)>0 \quad\text{for all } x\in\XX\setminus\{0\}
			\quad\iff\quad
			\bigcap_{i=1}^m \ker Q_i\cap \YY_-=\{0\}.
			\]
		\end{example}
	\end{mybox}
	
	The next example illustrates that the existence of a stationary minimizer of \eqref{eq:P} is not necessary 
	for properness and lower semicontinuity of $\mathcal{V}$.
	
	\begin{mybox}
		\begin{example}\label{example:cvx_deg}
			Consider \eqref{eq:P} with $g:=\indicator_{\R_-}$ as well as $\func{f}{\R}{\R}$ and $\func{c}{\R}{\R}$ given by
			\[
			f(x)\coloneqq x,\qquad c(x)\coloneqq x^2.
			\]
			The feasible set of this problem is $\{0\}$ and its solution is $\bar x\coloneqq 0$,
			which is not stationary.
			With $g^\conj=\indicator_{\R_+}$ we find $\bar y \coloneqq \nicefrac 12\in \dom g^\conj$, 
			and $\inf \left\{f + \innprod{\bar y}{c}\right\}= -\nicefrac 12>-\infty$. 
			Now, \cref{lem:value_function} guarantees that $\mathcal V$ is proper and lsc.
			In fact, we have $\dom\mathcal V=\R_+$ and $\mathcal V(u)=-\sqrt u$ for all $u\in\R_+$. 
		\end{example}
	\end{mybox}
	
	\subsection{Duality}
	
	We now study duality relations between \eqref{eq:P} and some appropriate dual problem.
	Of course, these results can, in more abstract form, be already found in \cite[Sec.\ 30]{Rockafellar1970}
	for large parts.
	However, we derive the essentials here in order to provide a self-contained presentation.
	
	To start, let us introduce a suitable dual problem associated with \eqref{eq:P}.
	Let $\func{\dualLL}{\YY}{\R\cup\{-\infty\}}$ be defined by
	\begin{equation}\label{eq:PPA:dualL}
		\dualLL(y)\coloneqq\inf_x \LL(x,y).
	\end{equation}
	By definition and \cref{ass:convexP}, $\dualLL$ indeed only takes values in $\R\cup\{-\infty\}$.
	Recalling \eqref{eq:saddleLL}, we refer to
	\begin{equation}
		\maximize_y ~ \dualLL(y)
		\label{eq:PPA:dual_problem}\tag{D}
	\end{equation}
	as the (Lagrange) \emphdef{dual problem} of \eqref{eq:P}.
	
	\begin{mybox}
		\begin{lemma}\label{lem:duality}
			Consider \eqref{eq:P} and \eqref{eq:PPA:dual_problem} under \cref{ass:convexP}.
			\begin{enumerate}[label=(\roman{*})]
				\item\label{item:weak_duality} 
				For $x\in \XX$ and $y\in\YY$, we always have $\Phi(x)\geq \dualLL(y)$.
				Particularly, the infimal value of \eqref{eq:P} is not smaller than the supremal value of \eqref{eq:PPA:dual_problem},
				and if $\Phi(x)=\dualLL(y)$, then $x$ is a minimizer of \eqref{eq:P}
				and $y$ is a maximizer of \eqref{eq:PPA:dual_problem}.
				\item\label{item:strong_duality} 
				Assume that \eqref{eq:P} possesses a stationary minimizer.
				Then \eqref{eq:P} and \eqref{eq:PPA:dual_problem} possess solutions,
				and the optimal function values coincide.
			\end{enumerate}
		\end{lemma}
	\end{mybox}
	\begin{proof}
		Let us start to verify assertion \ref{item:weak_duality}.
		We only prove validity of the estimate $\Phi(x)\geq\dualLL(y)$.
		The additional statements are direct consequences of it.
		Whenever $\Phi(x)=\infty$ or $\dualLL(y)=-\infty$, $\Phi(x)\geq\dualLL(y)$ is trivially satisfied.
		Hence, let us assume that $\Phi(x)<\infty$ and $\dualLL(y)>-\infty$.
		This gives $c(x)\in\dom g$ and $y\in\dom g^\conj$, so that the Fenchel--Young inequality,
		see \cite[Prop.\ 11.3]{rockafellar1998variational}, yields
		\begin{align*}
			\dualLL(y)
			=
			\inf_{x'} \LL(x',y)
			={}&
			\inf_{x'}\left\{f(x')+\innprod{y}{c(x')}\right\}-g^\conj(y) \\
			\leq{}&
			\inf_{x'}\left\{f(x')+\innprod{y}{c(x')}\right\}+g(c(x))-\innprod{y}{c(x)} \\
			\leq{}&
			f(x)+\innprod{y}{c(x)}+g(c(x))-\innprod{y}{c(x)}
			=
			\Phi(x)
			.
		\end{align*}
		
		Let us now turn to assertion \ref{item:strong_duality}.
		The stationarity assumption guarantees the existence of 
		$\bar x\in\XX$ and $\bar y\in\partial g(c(\bar x))$ such that $\nabla_x\LL(\bar x,\bar y)=0$.
		This yields $c(\bar x)\in\dom g$, $\bar y\in\dom g^\conj$, and $\innprod{\bar y}{c(\bar x)}=g(c(\bar x))+g^\conj(\bar y)$,
		again due to \cite[Prop.\ 11.3]{rockafellar1998variational}. 
		Furthermore, due to the convexity of $\LL(\cdot,\bar y)$, which follows from
		\cref{lem:convex_setting}\,\ref{item:convex_scalarization},
		$\nabla_x\LL(\bar x,\bar y)=0$ yields that $\bar x$ is a minimizer 
		of the convex function $\LL(\cdot,\bar y)$,
		i.e., $\dualLL(\bar y)=\LL(\bar x,\bar y)$.
		Combining these facts, we obtain
		\[
		\dualLL(\bar y)
		=
		\LL(\bar x,\bar y)
		=
		f(\bar x) + \innprod{\bar y}{c(\bar x)} - g^\conj(\bar y)
		=
		f(\bar x) + g(c(\bar x))
		=
		\Phi(\bar x),
		\]
		so that assertion \ref{item:strong_duality} follows from \ref{item:weak_duality}.
	\end{proof}
	
	Let us note that \cref{lem:duality}\,\ref{item:weak_duality} and its proof work even in the absence of any convexity assumptions
	as, particularly, the exploited Fenchel--Young inequality holds for any function $g$ by definition of the conjugate.
	Observe that \cref{lem:duality}\,\ref{item:weak_duality} ensures that $\dualLL$ is globally upper bounded 
	as \eqref{eq:P} is assumed to be a feasible problem in \cref{ass:convexP}.
	Let us also mention that, whenever the convexity assumptions are removed,
	\cref{lem:duality}\,\ref{item:strong_duality} does not remain true in general.
	
	\begin{mybox}
		\begin{definition}\label{def:saddle_point}
			A pair $(\bar x,\bar y)\in\XX\times\YY$ is referred to as a \emphdef{saddle point} of $\LL$
			if the following condition holds for all $x\in\XX$ and $y\in\YY$:
			\[
			\LL(\bar x,y)\leq\LL(\bar x,\bar y)\leq\LL(x,\bar y)
			.
			\]
		\end{definition}
	\end{mybox}
	
	Let us note that whenever $(\bar x,\bar y)\in\XX\times\YY$ is a saddle point of $\LL$,
	then $\LL(\bar x,\bar y)$ is finite since $g^\conj$ is proper
	(by properness of $g$ and \cite[Thm~11.1]{rockafellar1998variational}).
	
	\begin{mybox}
		\begin{lemma}\label{lem:duality_saddle_point_form}
			For a pair $(\bar x,\bar y)\in\XX\times\YY$, the following statements are equivalent under \cref{ass:convexP}:
			\begin{enumerate}[label=(\roman{*})]
				\item\label{item:stationarity} 
				$\bar x$ is stationary for \eqref{eq:P} with associated Lagrange multiplier $\bar y$;
				\item\label{item:no_duality_gap}
				$\bar x$ is a minimizer of \eqref{eq:P}, $\bar y$ is a maximizer of \eqref{eq:PPA:dual_problem},
				and $\Phi(\bar x)=\dualLL(\bar y)$;
				\item\label{item:saddle_point}
				$(\bar x,\bar y)$ is a saddle point of $\LL$.
			\end{enumerate}
		\end{lemma}
	\end{mybox}
	\begin{proof}
		The equivalence follows from a chain of implications.
		
		\ref{item:stationarity}$\implies$\ref{item:no_duality_gap}:
		This implication has been validated in the proof of \cref{lem:duality}\,\ref{item:strong_duality}.
		
		\ref{item:no_duality_gap}$\implies$\ref{item:saddle_point}:
		The assumptions, \eqref{eq:saddleLL}, and \eqref{eq:PPA:dualL} guarantee
		\begin{equation*}
			\dualLL(\bar y)
			=
			\inf_{x'} \LL(x',\bar y)
			\leq
			\LL(\bar x,\bar y)
			\leq
			\sup_{y'}\LL(\bar x,y')
			=
			\Phi(\bar x)
			=
			\dualLL(\bar y).
		\end{equation*}
		Hence, for each pair $(x,y)\in\XX\times\YY$, we find
		\begin{equation*}
			\LL(\bar x,y)
			\leq
			\sup_{y'}\LL(\bar x,y')
			=
			\LL(\bar x,\bar y)
			=
			\inf_{x'}\LL(x',\bar y)
			\leq
			\LL(x,\bar y),
		\end{equation*}
		which shows that $(\bar x,\bar y)$ is a saddle point of $\LL$.
		
		\ref{item:saddle_point}$\implies$\ref{item:stationarity}:
		As $\bar y$ maximizes the concave function $\LL(\bar x,\cdot)$,
		it minimizes the convex function $-\LL(\bar x,\cdot)$.
		Due to Fermat's rule,
		this gives $c(\bar x)\in\partial g^\conj(\bar y)$ or, equivalently under \cref{ass:convexP},
		$\bar y\in\partial g(c(\bar x))$.
		Particularly, $\bar y\in\dom g^\conj$, and \cref{lem:convex_setting}\,\ref{item:convex_scalarization}
		guarantees convexity of $\LL(\cdot,\bar y)$.
		As this function is continuously differentiable,
		the fact that $\bar x$ is one of its minimizers is enough to find
		$\nabla_x\LL(\bar x,\bar y)=0$.
		Hence, $\bar x$ is stationary for \eqref{eq:P} with associated
		multiplier $\bar y$.
	\end{proof}
	
	\cref{lem:duality_saddle_point_form} yields a series of interesting
	consequences which we record in what follows.
	
	\begin{mybox}
		\begin{corollary}\label{cor:convex_setting_multiplier_set}
			Let \cref{ass:convexP} hold, and
			let $x,x'\in\XX$ be stationary points of \eqref{eq:P}.
			Then $\optiYY(x)=\optiYY(x')$.
		\end{corollary}
	\end{mybox}
	\begin{proof}
		It suffices to pick two arbitrary stationary points $x^i\in\XX$, $i=1,2$, and
		to show that, for $y^i\in \optiYY(x^i)$, $i=1,2$,
		we have $y^{3-i}\in \optiYY(x^i)$.
		Due to \cref{lem:duality_saddle_point_form}, for $i=1,2$, $(x^i,y^i)$
		is a saddle point of $\LL$, i.e.,
		for all $x\in\XX$ and $y\in\YY$, we have
		\[
		\LL(x^i,y)\leq\LL(x^i,y^i)\leq\LL(x,y^i).
		\]
		Particularly, for $i=1,2$, we find
		\begin{equation*}
			\LL(x^{i},y^{3-i})
			\leq
			\LL(x^i,y^i)
			\leq
			\LL(x^{3-i},y^{i})
			\leq
			\LL(x^{3-i},y^{3-i})
			\leq
			\LL(x^{i},y^{3-i})
		\end{equation*}
		which gives
		$
		\LL(x^1,y^1)=\LL(x^2,y^1)=\LL(x^1,y^2)=\LL(x^2,y^2)
		$.
		Consequently, for $i=1,2$, $(x^{i},y^{3-i})$ is a saddle point of $\LL$,
		and \cref{lem:duality_saddle_point_form} yields $y^{3-i}\in \optiYY(x^i)$,
		concluding the proof.
	\end{proof}
	
	\begin{mybox}
		\begin{corollary}\label{cor:all_minimizers_stationary}
			Let \cref{ass:convexP} hold,
			and assume that \eqref{eq:P} possesses a stationary minimizer.
			Then each minimizer of \eqref{eq:P} is stationary.
		\end{corollary}
	\end{mybox}
	\begin{proof}
		Let $\bar x\in\XX$ be a stationary minimizer of \eqref{eq:P},
		and pick another minimizer $\tilde x\in\XX$ of \eqref{eq:P}.
		The assumptions guarantee that we find some $\bar y\in\optiYY(\bar x)$.
		Noting that $(\bar x,\bar y)$ is a saddle point of $\LL$ according to \cref{lem:duality_saddle_point_form},
		we have $\LL(\bar x,\bar y)\leq\LL(\tilde x,\bar y)$.
		Conversely, $\Phi(\bar x)=\Phi^\star=\Phi(\tilde x)$ and $\bar y\in\partial g(c(\bar x))$ yield
		\begin{multline*}
			\LL(\tilde x,\bar y)
			=
			f(\tilde x) + \innprod{\bar y}{c(\tilde x)} - g^\conj(\bar y)
			=
			f(\bar x) + g(c(\bar x))-g(c(\tilde x)) + \innprod{\bar y}{c(\tilde x)} - g^\conj(\bar y)
			\\
			\leq
			f(\bar x) + \innprod{\bar y}{c(\bar x)} - g^\conj(\bar y)
			=
			\LL(\bar x,\bar y).
		\end{multline*}
		Hence, $\LL(\tilde x,\bar y)=\LL(\bar x,\bar y)$, 
		and $\LL(\tilde x,\bar y)\leq\LL(x,\bar y)$ for all $x\in\XX$ follows.
		
		From \cref{lem:duality_saddle_point_form} we also have $\LL(\bar x,\bar y)=\Phi(\bar x)=\Phi^\star$,
		so that \eqref{eq:saddleLL} gives
		\[
		\sup_y\LL(\tilde x,y)
		=
		\Phi(\tilde x)
		=
		\Phi^\star
		=
		\LL(\bar x,\bar y)
		=
		\LL(\tilde x,\bar y).
		\]
		Hence, $(\tilde x,\bar y)$ is a saddle point of $\LL$,
		and applying \cref{lem:duality_saddle_point_form} one last time yields 
		that $\tilde x$ is stationary (with associated multiplier $\bar y$).
	\end{proof}
	
	\cref{cor:convex_setting_multiplier_set,cor:all_minimizers_stationary} show that, 
	for \eqref{eq:P} in the presence of \cref{ass:convexP},
	either all minimizers are stationary and share the same set of multipliers
	or all minimizers are nonstationary.
	Hence, in order to simplify notation, 
	in the situation where \eqref{eq:P} possesses a stationary minimizer $\bar x\in\XX$,
	let us introduce 
	\[
	\multYY\coloneqq\optiYY(\bar x),
	\]
	and emphasize that all minimizers of \eqref{eq:P} are stationary with multiplier set $\multYY$
	in this case.
	
	\begin{mybox}
		\begin{corollary}\label{cor:solution_set_dual_problem}
			Let \cref{ass:convexP} hold,
			and assume that \eqref{eq:P} possesses a stationary minimizer.
			Then the set of solutions associated with \eqref{eq:PPA:dual_problem} 
			is $\multYY$.
		\end{corollary}
	\end{mybox}
	\begin{proof}
		Let $\bar x\in\XX$ be a minimizer of \eqref{eq:P}.
		Then it is stationary according to \cref{cor:all_minimizers_stationary}
		with multiplier set $\multYY$.
		It is clear from \cref{lem:duality_saddle_point_form} that all vectors in 
		the nonempty set $\multYY$
		are solutions of \eqref{eq:PPA:dual_problem}.
		Let us fix $\bar y\in \multYY$ arbitrarily.
		Pick any $y^\prime\in\YY$ which is a solution of \eqref{eq:PPA:dual_problem}.
		Then $\dualLL(y^\prime)=\dualLL(\bar y)=\Phi(\bar x)$ follows
		from \cref{lem:duality_saddle_point_form}, 
		and this result also yields $y^\prime\in \multYY$.
		Hence, the solution set of \eqref{eq:PPA:dual_problem} is precisely $\multYY$.	
	\end{proof}
	
	In the lemma below, we summarize essential properties of the (negative) dual objective function.
	
	\begin{mybox}
		\begin{lemma}\label{lem:properties_dualLL}
			Let \cref{ass:convexP} be valid.
			Then the following assertions hold.
			\begin{enumerate}[label=(\roman{*})]
				\item The function $-\dualLL$ is convex, lsc, and bounded below by $-\Phi^\star$.
				\item\label{item:dualLL_vs_value_function} 
				We have $-\dualLL(y) = \mathcal{V}^\conj(-y)$ for all $y\in\YY$, 
				where $\mathcal V$ is the value function from \eqref{eq:value_function}.
				\item If, additionally, $\mathcal V$ is proper and lsc
				(see \cref{lem:value_function}\,\ref{lem:value_function:proper} for some sufficient conditions),
				then $-\dualLL$ is also proper.
			\end{enumerate}
		\end{lemma}
	\end{mybox}
	\begin{proof}
		The first assertion is clear by definition of $\dualLL$ and \cref{ass:convexP}.
		The lower boundedness of $-\dualLL$ directly follows
		from \cref{lem:duality_saddle_point_form}\,\ref{item:weak_duality},
		which yields $-\dualLL(y)\geq-\Phi^\star>-\infty$
		for each $y\in\YY$.
		
		Thus, let us proceed by proving the second assertion.	
		Using the substitution $z\coloneqq c(x)-u$ as well as the definition of the value function $\mathcal V$,
		we find
		\begin{align*}
			\dualLL(y)
			=
			\inf_x \LL(x,y)
			=&
			\inf_{x,z} \left\{ f(x)+g(z)+\innprod{y}{c(x)-z} \right\}
			={}
			\inf_{x,u} \left\{ f(x) + g(c(x)-u) + \innprod{y}{u} \right\}
			\\
			={}&
			\inf_u \left\{ \mathcal{V}(u)+\innprod{y}{u} \right\} 
			=
			- \sup_u \left\{ \innprod{-y}{u} - \mathcal{V}(u) \right\}
			=
			- \mathcal{V}^\conj(-y)
		\end{align*}
		for all $y\in\YY$. 
		
		Finally, let us validate the last assertion.
		To prove the properness of $-\dualLL$
		we merely need to show that $\dom(-\dualLL)$ is nonempty.
		Supposing that $-\dualLL\equiv\infty$, we find $\mathcal{V}^\conj\equiv-\infty$ 
		from assertion \ref{item:dualLL_vs_value_function}
		and, thus, $\mathcal{V}^{\conj\conj}\equiv\infty$.
		Then \cite[Thm~12.2]{Rockafellar1970},
		which requires that $\mathcal V$ is proper and lsc, 
		yields $\mathcal V\equiv\infty$.
		This, however, is a contradiction as we have $0 \in \dom \mathcal V$
		since the unperturbed problem \eqref{eq:P} is assumed to possess a minimizer.
	\end{proof}
	
	The following result, slightly weaker than \cref{cor:solution_set_dual_problem}, 
	relies on the characterizations in \cref{lem:properties_dualLL} and will be important in the course of the paper.
	
	\begin{mybox}
		\begin{lemma}\label{lem:solutions_of_dual_are_multipliers}
			Let \cref{ass:convexP} hold.
			Suppose that \eqref{eq:PPA:dual_problem} possesses a maximizer $\bar y\in\YY$,
			and let $\mathcal V$ be proper and lsc.
			Then each minimizer of \eqref{eq:P} is stationary,
			and $\bar y\in\multYY$ holds.
		\end{lemma}
	\end{mybox}
	\begin{proof}
		As $\bar y$ is a minimizer of the convex function $-\dualLL$, we find $0\in\partial(-\dualLL)(\bar y)$.
		We apply \cref{lem:properties_dualLL}\,\ref{item:dualLL_vs_value_function} 
		and \cite[Prop.\ 11.3]{rockafellar1998variational},
		the latter requiring that $\mathcal V$ is proper and lsc,
		to find
		\[
		0 \in \partial(-\dualLL)(\bar y)
		\quad\iff\quad
		0 \in \partial\mathcal{V}^\conj(-\bar y)
		\quad\iff\quad
		-\bar y \in \partial\mathcal{V}(0).
		\]
		By definition of the value function from \eqref{eq:value_function},
		$\overline\Phi$ from \eqref{eq:composite_mapping_perturbed},		
		and the subdifferential, we obtain
		\[
		\overline{\Phi}(u,x)
		\geq
		\mathcal V(u)
		\geq
		\mathcal V(0) - \innprod{\bar y}{u}
		=
		\overline{\Phi}(0,\bar x) - \innprod{\bar y}{u}
		\]
		for all $x\in\XX$, $u\in\YY$, and each minimizer $\bar x\in \XX$ of \eqref{eq:P}.
		This, however, means that $(-\bar y,0)\in\partial\overline{\Phi}(0,\bar x)$.
		Observe that the composite model function $\overline{\Phi}$ 
		(in variables $(u,x)$, see \eqref{eq:perturbed_problem_composite_setting})
		satisfies the qualification condition \eqref{eq:QC}
		at each point of its domain
		due to the full perturbation of the argument of $g$,
		and, hence, we find
		\[
		\partial\overline{\Phi}(u,x)
		\subseteq\{(-y,\nabla_x\LL(x,y))\,|\,y\in\partial g(c(x)-u)\}
		\] 
		for all $x\in\XX$ and $u\in\YY$ such that $c(x)-u\in\dom g$.
		Hence, condition $(-\bar y,0)\in\partial \overline{\Phi}(0,\bar x)$
		yields $\bar y\in\optiYY(\bar x)$.
		The result, thus, is a consequence of \cref{cor:convex_setting_multiplier_set}.
	\end{proof}

	\section{Classical ALM through the PPA lens}
	\label{sec:classical_ALM}
	
	Consider the ALM stated in \cref{alg:ALMclassic}.
	It employs the so-called \emphdef{Hestenes--Powell--Rockafellar} update rule for the multiplier in \cref{step:ALMclassic:y},
	see \cite{hestenes1969multiplier,powell1969method,rockafellar1973dual},
	and allows for a variety of update rules for the penalty parameter.
	
	We recall that, by convexity of $g$, $\LL_\mu(\cdot,y)$ is a continuously differentiable function
	for each $\mu>0$.
	In the convex environment considered in this paper,
	due to \cref{ass:convexP},
	it is reasonable to pursue a global minimization approach in the spirit of \cite[Ch.\ 5]{birgin2014practical},
	demanding that the augmented Lagrangian subproblems at \cref{step:ALMclassic:subproblem} are solved according to
	\begin{equation}\label{eq:subproblem_solution}
		\LL_{\mu_k}(x^k,y^k)
		\leq 
		\inf\LL_{\mu_k}(\cdot,y^k)+\varepsilon_k
	\end{equation}
	for some bounded or null sequence $\{\varepsilon_k\}$ of nonnegative inexactness parameters.
	At the first glance, 
	this condition seems to be challenging to test in numerical practice.
	However, it is a rather weak criterion, 
	and its validity can be checked in terms of valid upper bounds
	as mentioned and illustrated in \cite[Sec.~4]{rockafellar1976augmented}.

	\begin{algorithm2e}[htb]
		\DontPrintSemicolon
		\KwData{$\mu_0>0$; $y^0\in \YY$; $\beta\in(0,1)$\;}
		\For{$k = 0,1,2\ldots$}{
			Select $\varepsilon_k\geq 0$ and compute $x^k\in \XX$ such that \eqref{eq:subproblem_solution} holds.
			\label{step:ALMclassic:subproblem}%
			\;
			Set 
			\label{step:ALMclassic:z}%
			$z^k\coloneqq \prox_{\mu_k g}(c(x^k) + \mu_k y^k)$.%
			\;
			Set 
			\label{step:ALMclassic:y}%
			$y^{k+1} \coloneqq y^k + \mu_k^{-1} [c(x^k) - z^k]$.%
			\;
			Set $\mu_{k+1} \coloneqq \beta_k \mu_k$ for some $\beta_k\in(0,\beta]\cup\{1\}$.
		}
		\caption{Classical ALM for \eqref{eq:P}.}
		\label{alg:ALMclassic}
	\end{algorithm2e}

	Note that the sequence $\{\mu_k\}$ of penalty parameters in \cref{alg:ALMclassic}
	is monotonically nonincreasing and stays bounded.
	Hence, it converges (to some point in $[0,\mu_0]$).
	Forthwith, we are concerned with three well-known update rules for the penalty parameter
	in \cref{alg:ALMclassic}:
	\begin{mybox}
		\begin{setting}\label{set:constant_penalty_parameter}
			We have $\beta_k\coloneqq 1$ for all $k\in\N_0$.
		\end{setting}
		\begin{setting}\label{set:strictly_decreasing_penalty_parameter}
			There is some $\beta\in(0,1)$ such that $\beta_k\in(0,\beta]$ for each $k\in\N_0$.
		\end{setting}
		\begin{setting}\label{set:algencan_penalty_parameter}
			There are some $\beta,\theta\in(0,1)$ such that, for each $k\in\N_0$,
			\[
			\mu_{k+1}
			\coloneqq 
			\begin{cases}
				\mu_k	&	\text{if }k=0\text{ or }\norm{c(x^k)-z^k}\leq\theta\norm{c(x^{k-1})-z^{k-1}},
				\\
				\beta\,\mu_k	&	\text{otherwise.}
			\end{cases}
			\]
		\end{setting}
	\end{mybox}
	Note that $\mu_k$ remains constant in \cref{set:constant_penalty_parameter},
	i.e., $\mu_k=\mu_0$ holds for all $k\in\N$, whereas
	$\mu_k\downtoneq 0$ in \cref{set:strictly_decreasing_penalty_parameter}.
	The update rule in \cref{set:algencan_penalty_parameter} is implemented 
	in the popular solver \textsc{Algencan}, see \cite{andreani2008augmented,birgin2014practical}.
	Moreover, these settings cover also the globally convergent scheme of Conn, Gould, and Toint \cite{conn1991globally}, 
	known as \emphdef{bound-constrained Lagrangian} (BCL) algorithm and implemented in the \textsc{Lancelot} software package.
	Their technique monitors the constraint violation and updates the Lagrange multiplier estimates 
	only if the primal update achieves sufficient improvement in feasibility, 
	otherwise the penalty parameter is reduced, see \cite[Alg.\ 1]{conn1991globally}.
	For our purposes, the always-updating \cref{alg:ALMclassic} can be seen as executing the BCL scheme, 
	after extracting the subsequence with updates of multiplier estimates.

	\subsection{Dual updates as inexact PPA}
	
	When applied to the minimization of the convex function $-\dualLL$,
	\cref{alg:PPA} can be defined by the recursion
	\begin{equation}\label{eq:ALM_as_PPA}
		y^{k+1}
		\in
		\prox_{- \nicefrac{1}{\mu_k} \dualLL}^{\varepsilon_k}( y^{k} )
	\end{equation}
	with an upper bounded sequence of (inverse) positive stepsizes $\{\mu_k\}$.
	In the upcoming theorem,
	we will see that this iterative procedure is strongly related to the ALM from \cref{alg:ALMclassic}.
	Our proof follows the one of \cite[Prop.\ 6]{rockafellar1976augmented},
	see \cite[Prop.\ 1.3]{andrews2025augmented} as well,
	and is added for the convenience of the reader.
	
	\begin{mybox}
		\begin{theorem}\label{thm:inexactPPA:saddle}
			Let $\hat{y} \in \YY$, $\mu>0$, and $\varepsilon \geq 0$ be fixed.
			Suppose that \cref{ass:convexP} is satisfied and that $\bar{x}\in\XX$ is chosen such that 
			$\LL_\mu(\bar{x},\hat{y}) \leq \inf \LL_\mu(\cdot,\hat{y}) + \varepsilon$.
			Set $\bar{z} \coloneqq \prox_{\mu g}(c(\bar{x}) + \mu \hat{y})$
			and $\bar{y} \coloneqq \hat{y} + \mu^{-1}[c(\bar{x}) - \bar{z}]$.
			Then $\bar{y}\in \prox_{- \nicefrac{1}{\mu} \dualLL}^\varepsilon(\hat{y})$.
		\end{theorem}
	\end{mybox}
	\begin{proof}
		For later use, let us denote the exact proximal point by
		$
		y_{\rm e}\coloneqq \prox_{-\nicefrac{1}{\mu}\dualLL}(\hat y).
		$
		Differentiability of the Moreau envelope yields
		\begin{equation}\label{eq:der_Moreau_envelope}
			\nabla_y\LL_\mu(\bar x,\hat y)
			=
			c(\bar x)+\mu\hat y - \bar z - \mu\hat y
			=
			c(\bar x) - \bar z
			=
			\mu(\bar y - \hat y),
		\end{equation}
		see \cite[Prop.\ 12.29]{BauschkeCombettes2011}.
		Furthermore, let us recall from \cref{lem:convex_setting} that $\LL_\mu(\bar x,\cdot)$ is concave.
		Hence, for some arbitrary $y\in\YY$, we find the estimate
		\begin{equation*}
			\LL_\mu(\bar x,\hat y) + \innprod{\mu(\bar y-\hat y)}{y-\hat y}
			=
			\LL_\mu(\bar x,\hat y) + \innprod{\nabla_y\LL_\mu(\bar x,\hat y)}{y-\hat y} 
			\geq
			\LL_\mu(\bar x,y)
			\geq
			\inf\LL_\mu(\cdot,y).
		\end{equation*}
		Furthermore, \cite[Lem.\ 36.1]{Rockafellar1970} 
		and the definition of $\dualLL$ give
		\begin{align*}
			\inf\LL_\mu(\cdot,y)
			\overset{\eqref{eq:rep_AL_func}}{=}{}
			\inf_x\sup_z\left\{\LL(x,z) - \frac{\mu}{2}\norm{z-y}^2\right\}
			\geq{}&
			\sup_z\inf_x\left\{\LL(x,z) - \frac{\mu}{2}\norm{z-y}^2\right\}
			\\
			\overset{\eqref{eq:PPA:dualL}}{=}{}&
			\sup_z\left\{\dualLL(z) - \frac{\mu}{2}\norm{z-y}^2\right\}
			\geq 
			\dualLL(y_{\rm e}) - \frac{\mu}{2}\norm{y_{\rm e}-y}^2,
		\end{align*}
		showing validity of the estimate
		\[
		\LL_\mu(\bar x,\hat y) + \innprod{\mu(\bar y-\hat y)}{y-\hat y}
		\geq
		\dualLL(y_{\rm e}) - \frac{\mu}{2}\norm{y_{\rm e}-y}^2
		\]
		for each $y\in\YY$.
		Similarly, using the definition of $y_{\rm e}$, we find
		\begin{align*}
			\inf\LL_\mu(\cdot,\hat y)
			\overset{\eqref{eq:rep_AL_func}}{=}{}
			\inf_x\sup_z\left\{\LL(x,z) - \frac{\mu}{2}\norm{z-\hat y}^2\right\}
			={}&
			\sup_z\inf_x\left\{\LL(x,z) - \frac{\mu}{2}\norm{z-\hat y}^2\right\}
			\\
			\overset{\eqref{eq:PPA:dualL}}{=}{}&
			\sup_z\left\{\dualLL(z) - \frac{\mu}{2}\norm{z-\hat y}^2\right\}
			=
			\dualLL(y_{\rm e}) - \frac{\mu}{2}\norm{y_{\rm e}-\hat y}^2
		\end{align*}
		with the aid of \cite[Thm~37.3]{Rockafellar1970},
		which applies as
		$\nicefrac{\mu}{2}\norm{\cdot -\hat y}^2-\LL(x,\cdot)$ is uniformly convex for each $x\in\XX$
		yielding that its level sets are bounded and, thus,
		that its recession function is the indicator function of the origin,
		see \cite[Thm~8.7]{Rockafellar1970}.
		Hence, for each $y\in\YY$, the above and some simple calculations yield
		\begin{align*}
			\varepsilon
			\geq{}&
			\LL_\mu(\bar x,\hat y) - \inf\LL_\mu(\cdot,\hat y)
			\\
			\geq{}&
			\frac{\mu}{2}\left(
			\norm{y_{\rm e}-\hat y}^2 - 2\innprod{\bar y-\hat y}{y-\hat y} - \norm{y_{\rm e}-y}^2
			\right)
			=
			\frac{\mu}{2}\left(2\innprod{y_{\rm e}-\bar y}{y-\hat y}-\norm{y-\hat y}^2\right).
		\end{align*}
		Choosing $y\coloneqq y_{\rm e}-\bar y+\hat y$, the right-hand side is maximized, 
		and we end up with the estimate $\varepsilon\geq\nicefrac{\mu}{2}\norm{y_{\rm e}-\bar y}^2$
		yielding the claim by definition of $y_{\rm e}$ and \eqref{eq:inexact_prox_mapping}.
	\end{proof}
	
	Equipped with \cref{thm:inexactPPA:saddle}, convergence guarantees for \cref{alg:ALMclassic} can be established 
	based on its equivalence to the inexact PPA in \cref{alg:PPA} in the form \eqref{eq:ALM_as_PPA} 
	thanks to \cref{lem:PPA} whenever $-\dualLL$ is proper.
	If $-\dualLL$ is not proper (meaning that $-\dualLL\equiv\infty$),
	then the value function $\mathcal V$ from \eqref{eq:value_function} 
	is likely to be not proper as well, see \cref{lem:properties_dualLL},
	and the subproblems under consideration must be unbounded in this situation,
	see \cite[Appx~B]{andrews2025augmented}.
	Hence, in order to ensure properness of $-\dualLL$ and, thus, justify an application of \cref{alg:ALMclassic},
	we will assume throughout that the value function $\mathcal V$ from \eqref{eq:value_function}
	is proper and lsc, see \cref{lem:properties_dualLL} again.
	According to \cref{lem:value_function},
	these properties are inherent whenever \eqref{eq:P} possesses a stationary minimizer.
	However, \cref{lem:value_function} also provides sufficient criteria for the validity
	of these properties in the irregular situation where each minimizer of \eqref{eq:P}
	is nonstationary.
	In our convergence analysis, both situations will be handled, but separately.

	\subsection{Convergence analysis}
	
	Now, we are in a position to generalize the classical convergence result 
	from \cite[Thm~2.1]{Rockafellar1973}, see \cite[Thm~4]{rockafellar1976augmented} as well,
	which addresses standard convex nonlinear optimization problems with inequality constraints,
	to the fully convex composite setting.
	
	In the light of \cref{lem:PPA}, it is sufficient to choose the sequence $\{\varepsilon_k\}$ 
	such that $\{ \sqrt{\varepsilon_k/\mu_k} \}$ is summable 
	in order to find the convergence of $\{y^k\}$
	provided that \eqref{eq:P} possesses a stationary minimizer.
	Note that constructing such a sequence is always possible; 
	among others, a trivial choice is $\varepsilon_k\coloneqq \nicefrac{\mu_k}{2^{2k}}$ for all $k\in\N_0$.
	\cref{thm:global_convergence} summarizes the analysis with inexact subsolves 
	and general nonincreasing sequence $\{\mu_k\}$ of penalty parameters.
	Its proof invokes the estimate 
	\begin{equation}\label{eq:bound_auglag_optivalue}
		\LL_{\mu_k}(x^k,y^k)
		\leq
		\inf\LL_{\mu_k}(\cdot,y^k)+\varepsilon_k
		\leq
		\inf \Phi + \varepsilon_k
		=
		\Phi^\star + \varepsilon_k,
	\end{equation}
	which follows from \eqref{eq:subproblem_solution} and \cite[Lem.\ 3.2]{DeMarchiMehlitz2024},
	and holds regardless of the feasibility of \eqref{eq:P}.
	
	\begin{mybox}
		\begin{theorem}\label{thm:global_convergence}
			Suppose that \cref{ass:convexP} holds and \eqref{eq:P} possesses a stationary minimizer. 
			Let $\{(x^k,z^k,y^k)\}$ be a sequence generated by \cref{alg:ALMclassic}
			such that
			\begin{equation}\label{eq:summability}
				\sum_{k=0}^\infty\sqrt{\varepsilon_k/\mu_k} < \infty.
			\end{equation}
			Then $\{y_k\}$ converges to some vector in $\multYY$,
			and we have $c(x^k)-z^k\to 0$ and
			\begin{equation}\label{eq:upper_limit_optimal}
				\limsup_{k\to\infty}\left( f(x^k)+g(z^k) \right)\leq \Phi^\star.
			\end{equation}
			Particularly, each accumulation point $\bar x\in\XX$ of $\{x^k\}$ is a minimizer of \eqref{eq:P}.
			For each subsequence $\{x^k\}_{k\in K}$ such that $x^k\to_K\bar x$,
			we have $z^k\to_K c(\bar x)$ and $f(x^k)+g(z^k)\to_K\Phi^\star$.
		\end{theorem}
	\end{mybox}
	\begin{proof}
		From \cref{lem:PPA,thm:inexactPPA:saddle} we know that $\{y^k\}$ converges
		to some maximizer $\bar y\in\YY$ of \eqref{eq:PPA:dual_problem},
		and $\bar y$ belongs to $\multYY$, see \cref{cor:solution_set_dual_problem}.
		
		Next, we observe that
		\[
		c(x^k)-z^k
		=
		\mu_k (y^{k+1}-y^k)
		\to
		0
		\]
		holds due to the convergence of $\{y^k\}$
		and the fact that $\{\mu_k\}$ is bounded.
		Furthermore, \eqref{eq:bound_auglag_optivalue} gives
		\begin{multline*}
			f(x^k)+g(z^k)+\innprod{y^k}{c(x^k)-z^k}
			\leq
			f(x^k)+g(z^k)+\innprod{y^k}{c(x^k)-z^k}+\frac1{2\mu_k}\norm{c(x^k)-z^k}^2
			\\
			=
			\LL_{\mu_k}(x^k,y^k)
			\leq
			\Phi^\star+\varepsilon_k.
		\end{multline*}
		Passing to the upper limit as $k\to\infty$ while using $c(x^k)-z^k\to 0$, $y^k\to\bar{y}$,
		and $\varepsilon_k\downto 0$ (the latter follows from \eqref{eq:summability} and boundedness of $\{\mu_k\}$),
		we find \eqref{eq:upper_limit_optimal}.
		
		Finally, let $\bar x\in\XX$ be the limit of the subsequence $\{x^k\}_{k\in K}$ 
		for some infinite set $K\subseteq\N$.
		Then $c(x^k)\to_K c(\bar x)$, and, due to $c(x^k)-z^k\to 0$, $z^k\to_K c(\bar x)$ follows.
		Hence, \eqref{eq:upper_limit_optimal} and the lower semicontinuity of $g$ yield
		\[
		\Phi(\bar x)
		\leq
		\lim_{k\to_K\infty}f(x^k)+\liminf_{k\to_K\infty}g(z^k)
		\leq
		\limsup_{k\to\infty}\left( f(x^k)+g(z^k) \right)
		\leq
		\Phi^\star.
		\]
		As this implicitly requires $\bar x\in c^{-1}(\dom g)$, 
		optimality of $\bar x$ for \eqref{eq:P} follows.
	\end{proof}
	
	Let us note that, in the setting of \cref{thm:global_convergence},
	each accumulation point of $\{x^k\}$ is in fact a stationary minimizer
	of \eqref{eq:P} and the limit of $\{y^k\}$ is an associated multiplier,
	see \cref{cor:convex_setting_multiplier_set,cor:all_minimizers_stationary}.
	
	\begin{mybox}
		\begin{remark}\label{rem:ks_example_nonconvex}
			In \cite{kanzow2017example}, Kanzow and Steck consider \eqref{eq:P} with $g\coloneqq\indicator_{\R_-}$
			as well as $\func{f}{\R}{\R}$ and $\func{c}{\R}{\R}$ given by
			\[
			f(x)\coloneqq x,\qquad c(x)\coloneqq 1-x^3.
			\]
			The feasible set of the associated problem \eqref{eq:P} equals the convex set $[1,\infty)$,
			so \eqref{eq:P} is a convex optimization problem.
			Its unique minimizer $\bar x\coloneqq 1$ is stationary with multiplier $\bar y\coloneqq\nicefrac13$.
			The considerations in \cite[\S 3.1]{kanzow2017example} illustrate that \cref{alg:ALMclassic}
			in \cref{set:algencan_penalty_parameter} generates a primal sequence with an infeasible accumulation point.
			
			We note that the functional description of the associated problem \eqref{eq:P} involves
			the nonconvex function $c$, so \cref{ass:convexP} is violated here as $\hzn g = \R_-$.
			As \cref{thm:global_convergence} shows, there does not exist an instance of \eqref{eq:P}
			satisfying \cref{ass:convexP} and possessing a stationary minimizer such that \cref{alg:ALMclassic}
			generates a sequence with infeasible primal accumulation points.
		\end{remark}
	\end{mybox}
	
	It may happen that \eqref{eq:P} merely possesses nonstationary minimizers,
	in which case \cref{thm:global_convergence} does not provide any information.
	However, from \cref{lem:PPA} it is still possible to infer primal convergence guarantees.
	Note that, by \cref{cor:all_minimizers_stationary}, requiring that all minimizers
	are (non)stationary is not restrictive.
	
	\begin{mybox}
		\begin{theorem}\label{thm:global_convergence_irregular}
			Suppose that \cref{ass:convexP} holds.
			Furthermore, let all minimizers of \eqref{eq:P} be nonstationary,
			and let the value function $\mathcal V$ from \eqref{eq:value_function}
			be proper and lsc.
			Let $\{(x^k,z^k,y^k)\}$ be a sequence generated by \cref{alg:ALMclassic}
			such that \eqref{eq:summability} holds.
			Then we have $c(x^k)-z^k\to 0$, $\norm{y^k}\to \infty$, and
			\begin{equation}\label{eq:lower_limit_optimal}
				\liminf_{k\to\infty}\left( f(x^k)+g(z^k) \right)\leq \Phi^\star.
			\end{equation}
			If $\{x^k\}$ is bounded, then there exist an accumulation point $\bar x\in\XX$
			and a subsequence $\{x^k\}_{k\in K}$ such that $x^k\to_K\bar x$,
			$\bar x$ is a minimizer of \eqref{eq:P}, and we have
			$z^k\to_K c(\bar x)$ and $f(x^k)+g(z^k)\to_K\Phi^\star$.
		\end{theorem}
	\end{mybox}
	\begin{proof}
		Together with \cref{lem:properties_dualLL,lem:solutions_of_dual_are_multipliers},
		\cref{lem:PPA} yields $\norm{y^k}\to\infty$ and $\mu_k(y^{k+1}-y^k)\to 0$.
		Hence, by construction of \cref{alg:ALMclassic}, $c(x^k)-z^k\to 0$ follows.
		Relation \eqref{eq:bound_auglag_optivalue} gives
		\begin{multline*}
			f(x^k)+g(z^k)+\frac{\mu_k}{2}\left( \norm{y^{k+1}}^2-\norm{y^k}^2 \right)
			=
			f(x^k)+g(z^k)+\frac{1}{2\mu_k}\norm{c(x^k)+\mu_k y^k-z^k}^2-\frac{\mu_k}{2}\norm{y^k}^2
			\\
			=
			\LL_{\mu_k}(x^k,y^k)
			\leq
			\Phi^\star+\varepsilon_k.
		\end{multline*}
		From $\norm{y^k}\to\infty$
		we infer the relation $\limsup_{k\to\infty}(\norm{y^{k+1}}^2-\norm{y^k}^2)\geq 0$,
		and, thus, $\limsup_{k\to\infty}\mu_k(\norm{y^{k+1}}^2-\norm{y^k}^2)\geq 0$.
		Hence, taking the upper limit in the above estimate yields \eqref{eq:lower_limit_optimal},
		together with $\varepsilon_k\downto 0$ from \eqref{eq:summability}:
		\begin{align*}
			\Phi^\star
			\geq{}&
			\limsup_{k\to\infty}\left(
			f(x^k)+g(z^k)+\frac{\mu_k}{2}\left( \norm{y^{k+1}}^2-\norm{y^k}^2 \right)
			\right)
			\\
			\geq{}&
			\liminf_{k\to\infty}\left( f(x^k)+g(z^k) \right)
			+
			\limsup_{k\to\infty}\frac{\mu_k}{2}\left( \norm{y^{k+1}}^2-\norm{y^k}^2 \right)
			\geq
			\liminf_{k\to\infty}\left( f(x^k)+g(z^k) \right)
			.
		\end{align*}
		
		For the remainder of the proof,
		we assume that $\{x^k\}$ is bounded.
		Let us pick an infinite set $K'\subseteq\N$ such that 
		\[
		f(x^k) + g(z^k) \to_{K'} \liminf_{k\to\infty}\left( f(x^k)+g(z^k) \right).
		\]
		Since the subsequence $\{x^k\}_{k\in K'}$ must be bounded,
		we find an infinite set $K\subseteq K'$ and $\bar x\in\XX$ such that $x^k\to_{K}\bar x$.
		Hence, $c(x^k)-z^k\to 0$ and the continuity of $c$ guarantee $z^k\to_K c(\bar x)$.
		Now, lower semicontinuity of $g$ and \eqref{eq:lower_limit_optimal} yield
		\[
		\Phi(\bar x)
		\leq
		\lim_{k\to_K\infty}\left( f(x^k)+g(z^k) \right)
		=
		\liminf_{k\to\infty}\left( f(x^k)+g(z^k) \right)
		\leq
		\Phi^\star.
		\]
		Particularly, this implies $\bar x\in c^{-1}(\dom g)$,
		and optimality of $\bar x$ for \eqref{eq:P} is obtained.
	\end{proof}
	
	Following the proof of \cref{thm:global_convergence_irregular},
	it is possible to obtain the stronger estimate \eqref{eq:upper_limit_optimal} whenever $\{\norm{y^k}\}$
	is monotonically nondecreasing or, at least, if
	\[
	\liminf_{k\to\infty}\mu_k\left(\norm{y^{k+1}}^2-\norm{y^k}^2\right)\geq 0
	\]
	is true (which is implied by the monotonicity of $\{\norm{y^k}\}$).
	This question addresses the analysis of (inexact) PPA in situations where the produced sequence diverges,
	and might be of independent interest.
	Observe that validity of \eqref{eq:upper_limit_optimal} would allow 
	to obtain the qualitative results of \cref{thm:global_convergence}
	even in the irregular setting considered in \cref{thm:global_convergence_irregular}.
	
	\begin{mybox}
		\begin{remark}\label{rem:primal_convergence_irregular_case}
			We note that the generality of \cref{thm:global_convergence_irregular}
			is not covered by the findings in \cite{andrews2025augmented}
			where $g$ is merely an indicator function associated
			to a given closed, convex set.
			Furthermore, in \cite{andrews2025augmented},
			the authors merely consider situations where $\mu_k\downarrow 0$
			and require further delicate control on $\{\mu_k\}$ to obtain their qualitative results.
		\end{remark}
	\end{mybox}

	\subsection{Dual updates as proximal-gradient ascent}
	
	It is known that the dual update in ALM corresponds to a (projected) gradient ascent step \cite[\S 2.1]{boyd2011distributed}.
	This observation is extended here to the setting of \eqref{eq:P},
	for which the dual update at \cref{step:ALMclassic:y} of \cref{alg:ALMclassic} 
	can be interpreted as a proximal-gradient ascent step with stepsize $\mu_k^{-1}$, as we now uncover.
	
	Specifically, noticing that $-\innprod{\cdot}{c(x)}$ and $g^\conj$ are the smooth and nonsmooth terms of $-\LL(x,\cdot)$, respectively, 
	the dual update $y^{k+1}$ turns out to be a proximal-gradient step to maximize the Lagrangian $\LL(x^k,\cdot)$ 
	starting from the current estimate $y^k$, see \eqref{eq:dual_update_as_prox_grad_ascent}.
	Alternatively, the dual update can also be recovered as the gradient step (again with stepsize $\mu_k^{-1}$) 
	to maximize the augmented Lagrangian $\LL_{\mu_k}(x^k,\cdot)$ from $y^k$, see \eqref{eq:dual_update_as_grad_ascent}.
	These observations are precisely stated in the following lemma.
	
	\begin{mybox}
		\begin{lemma}\label{thm:dual_update_as_prox_grad}
			Let $\bar{x}\in\XX$, $\hat{y} \in \YY$, and $\mu>0$ be fixed.
			Suppose that \cref{ass:convexP} is satisfied.
			Set $\bar{z} \coloneqq \prox_{\mu g}(c(\bar{x}) + \mu \hat{y})$
			and $\bar{y} \coloneqq \hat{y} + \mu^{-1}[c(\bar{x}) - \bar{z}]$.
			Then the following identities hold:
			\begin{subequations}
				\begin{align}\label{eq:dual_update_as_prox_grad_ascent}
					\bar{y}
					={}&
					\prox_{\nicefrac{1}{\mu} g^\conj}\left( \hat{y} + \frac{1}{\mu} c(\bar{x}) \right)
					, \\
					\bar{y}
					={}&
					\hat{y} + \frac{1}{\mu} \nabla_y \LL_\mu(\bar{x},\hat{y})
					.\label{eq:dual_update_as_grad_ascent}
				\end{align}
			\end{subequations}
		\end{lemma}
	\end{mybox}
	\begin{proof}
		The gradient ascent update in \eqref{eq:dual_update_as_grad_ascent} readily follows from the definitions of $\bar{z}$ and $\bar{y}$
		as well as \eqref{eq:der_Moreau_envelope}.
		Using \eqref{eq:dual_update_as_grad_ascent}, the proximal-gradient ascent update in \eqref{eq:dual_update_as_prox_grad_ascent} is an immediate consequence 
		of the (extended) Moreau decomposition, see \cite[Thm~6.45]{Beck2017}.
	\end{proof}

	\section{Safeguarded ALM through the PPA lens}
	\label{sec:safeguarded_ALM}
	
	In numerical optimization, so-called safeguarded ALMs are popular as they possess better
	global convergence properties than the classical method, see, e.g., \cite{kanzow2017example}.
	Particularly, the safeguarding mechanism allows us to prove (primal) convergence even in the nonconvex, nonregular, and/or nonsmooth case,
	see, e.g., \cite{birgin2014practical,demarchi2024implicit,demarchi2023constrained,DeMarchiMehlitz2024,KanzowKraemerMehlitzWachsmuthWerner2025}.
	Here, ``convergence'' is understood in the sense of asymptotic optimality or stationarity, even if multipliers do not exist.
	
	Motivated by \cref{thm:global_convergence}, we now aim to discuss the behavior of a prototypical 
	safeguarded ALM in the fully convex composite setting induced by \cref{ass:convexP}.
	The procedure of interest, taken from \cite{DeMarchiMehlitz2024}, is stated in \cref{alg:ALMsafeguarded}.
	The safeguarding technique employs a sequence $\{\hat y^k\}$ of multiplier estimates which
	remains bounded by construction.
	Analogously to \eqref{eq:subproblem_solution}, subproblems at \cref{step:ALMsafeguarded:subproblem} are solved according to
	\begin{equation}\label{eq:subproblem_solution:safeguarded}
		\LL_{\mu_k}(x^k,\hat{y}^k)
		\leq 
		\inf\LL_{\mu_k}(\cdot,\hat{y}^k)+\varepsilon_k.
	\end{equation}
	Let us mention that, similar to \cref{alg:ALMclassic}, 
	\cref{alg:ALMsafeguarded} is reasonable in 
	\cref{set:constant_penalty_parameter,set:strictly_decreasing_penalty_parameter,set:algencan_penalty_parameter}.
	However, convergence with constant penalty parameter as in \cref{set:constant_penalty_parameter} cannot be expected in general.
	Similarly, the safeguarding cannot be blindly applied
	but it should be accompanied by suitable updates of the penalty parameter.
	Both issues are illustrated in \cref{ex:safeguarded_ALM_with_constant_penalty_parameter} below. 
	
	\begin{algorithm2e}[htb]
		\DontPrintSemicolon
		\KwData{$y^0\in \YY$; $\mu_0>0$; nonempty, convex, compact set $\safeYY\subseteq\YY$; $\beta\in(0,1)$\;}
		\For{$k = 0,1,2\ldots$}{
			Set $\hat{y}^k\coloneqq \proj_{\safeYY}(y^k)$.
			\label{step:ALMsafeguarded:ysafe}%
			\;
			Select $\varepsilon_k\geq 0$ and compute $x^k\in \XX$ such that \eqref{eq:subproblem_solution:safeguarded} holds.
			\label{step:ALMsafeguarded:subproblem}%
			\;
			Set 
			\label{step:ALMsafeguarded:z}%
			$z^k\coloneqq \prox_{\mu_k g}(c(x^k) + \mu_k \hat y^k)$.%
			\;
			Set 
			\label{step:ALMsafeguarded:y}%
			$y^{k+1} \coloneqq \hat y^k + \mu_k^{-1} [c(x^k) - z^k]$.%
			\;
			Set $\mu_{k+1} \coloneqq \beta_k \mu_k$ for some $\beta_k\in(0,\beta]\cup\{1\}$.
		}
		\caption{Safeguarded ALM for \eqref{eq:P}.}
		\label{alg:ALMsafeguarded}
	\end{algorithm2e}
	
	It follows from \cref{thm:inexactPPA:saddle} that the safeguarded ALM, sketched in \cref{alg:ALMsafeguarded}, 
	can be viewed as a generalized PPA type method, applied to the dual problem \eqref{eq:PPA:dual_problem}, taking
	\begin{equation}\label{eq:generalized_PPA}
		y^{k+1} \in \prox_{- \nicefrac{1}{\mu_k} \dualLL}^{\varepsilon_k}(\hat{y}^k)
	\end{equation}
	for all $k\in\N_0$.
	In \cite{kanzow2017generalized},
	it has been observed that this scheme differs from the original PPA in that it takes $\hat{y}^k$ instead of $y^{k}$, 
	and such gap is due to the safeguarding---necessary for global primal convergence in the nonconvex case, see \cite{kanzow2017example} again.
	Furthermore, it has been shown in \cite{kanzow2017generalized} that each accumulation point of a sequence $\{y^k\}$
	computed via \eqref{eq:generalized_PPA} for bounded $\{\hat y^k\}$, $\varepsilon_k\coloneqq 0$ for each $k\in\N_0$,
	and $\mu_k\downtoneq 0$ is a maximizer of $\dualLL$.
	Here, we will deepen the insights into this generalized PPA significantly
	whenever $\hat y^k$ is obtained from $y^k$ by projection onto a 
	reasonably chosen safeguarding set $\safeYY$,
	which is the standard way to implement safeguarded ALMs,
	and exploit these findings for the convergence analysis of \cref{alg:ALMsafeguarded}.
	
	\begin{mybox}
		\begin{remark}\label{rem:safeguarding_set}
			Note that the goal behind (safeguarded) augmented Lagrangian methods is not only to find
			minimizers of \eqref{eq:P} as accumulation points of the produced primal sequence,
			but to construct associated multipliers (in case of existence)
			as accumulation points of the produced dual sequence.
			Given $\bar z\in\dom g$, one can easily check that $\partial g(\bar z)\subseteq (\hzn g)^\circ$ is valid, i.e.,
			the images of the set-valued mapping $\ffunc{\partial g}{\YY}{\YY}$ are contained in
			the polar cone of $\hzn g$.
			Insofar, observing that all possible multipliers of \eqref{eq:P},
			which are sufficiently small in norm,
			should belong to $\safeYY$,
			it makes sense to choose $\safeYY\coloneqq\{y\in(\hzn g)^\circ\,|\,\norm{y}\leq C\}$ 
			for some large constant $C>0$.
			Whenever $g$ is the indicator function of a closed, convex cone,
			this is, indeed, the preferred choice for $\safeYY$.
		\end{remark}
	\end{mybox}
	
	\subsection{Dual updates as inexact backward-backward splitting}
	
	Note that the dual iterates $\{y^k\}$ in \cref{alg:ALMsafeguarded} obey the recurrence
	\begin{equation}\label{eq:alternating_operators}
		y^{k+1}
		\in
		\prox_{- \nicefrac{1}{\mu_k} \dualLL}^{\varepsilon_k} \left( \proj_{\safeYY}(y^k) \right)
	\end{equation}
	for all $k\in\N_0$, see \cref{thm:inexactPPA:saddle}.
	Here, we do not discuss further the associated, though less interesting, recurrence 
	\begin{equation*}
		\hat{y}^{k+1}
		\in
		\proj_{\safeYY} \left( \prox_{- \nicefrac{1}{\mu_k} \dualLL}^{\varepsilon_k}(\hat{y}^k) \right)
	\end{equation*}
	on the surrogate sequence $\{\hat y^k\}$.
	
	In order to get a feeling of the behavior of the iteration \eqref{eq:alternating_operators},
	let us assume, for simplicity, that $\varepsilon_k=0$ for all $k\in\N_0$ and
	that \cref{set:constant_penalty_parameter} is chosen.
	Let $\mu\coloneqq\mu_0$ denote the associated constant penalty parameter.
	Noting that the operators $\proj_{\safeYY}$ and $\prox_{-\nicefrac1\mu\dualLL}$ are firmly nonexpansive,
	in this particular case,
	the iterative procedure from \eqref{eq:alternating_operators}
	is a special instance of the method discussed in \cite{ArizaLopezNicolae2015,banert2014backward},
	see \cite{AckerPrestel1980,BauschkeCombettesReich2005} as well.
	Following \cite[Thm~4.2]{ArizaLopezNicolae2015} 
	while noting that the image of $\proj_{\safeYY}$ is nonempty and compact,
	we obtain that \eqref{eq:alternating_operators} yields a sequence converging to a fixed point of
	$\prox_{- \nicefrac{1}{\mu} \dualLL}\circ\proj_{\safeYY}$ provided there is one.
	Let us note that the fixed-point condition 
	$\bar y = \prox_{- \nicefrac{1}{\mu} \dualLL}(\proj_{\safeYY}(\bar y))$ 
	is equivalent to the inclusion
	\begin{equation}
		\mu\left(\proj_{\safeYY}(\bar y)-\bar y\right)\in\partial(-\dualLL)(\bar y)
	\end{equation}
	by definition of the proximal operator and the convexity of $-\dualLL$.
	Using \cite[Cor.\ 12.30]{BauschkeCombettes2011}, 
	this is the same as postulating that $\bar y$ is a stationary
	point of the convex optimization problem
	\begin{equation}\label{eq:penalized_dual}\tag{D$_\mu$}
		\minimize_{y}
		{}\quad{}
		- \dualLL(y) + \frac{\mu}{2} \dist_{\safeYY}^2(y),
	\end{equation}
	i.e., one of its minimizers.
	Note that \eqref{eq:penalized_dual} can be interpreted as a penalized version
	of the dual problem \eqref{eq:PPA:dual_problem} complemented by the constraint $y\in \safeYY$.
	As \eqref{eq:P} is assumed to be feasible,
	the objective function of \eqref{eq:PPA:dual_problem} is lower bounded due to \cref{lem:duality},
	and boundedness of $\safeYY$ additionally yields that the objective function of \eqref{eq:penalized_dual} is coercive.
	Hence, this problem possesses a minimizer
	whenever $-\dualLL$ is proper and lsc
	which can be guaranteed via mild assumptions on the problem data,
	see \cref{lem:value_function,lem:properties_dualLL},
	so that \eqref{eq:alternating_operators} indeed produces a convergent sequence
	in this special situation.
	
	In a remark, Andreani et al.\ \cite[\S 5]{andreani2008augmented} ask the following question:
	``\emph{What happens if the box chosen for the safeguarded multiplier estimates is too small?}''
	Among other things, our streamlined convergence analysis in \cref{sec:convergence_rigid_safeguard} 
	rediscovers their answer from the PPA perspective:
	global convergence properties remain, but penalty parameters should vanish if the safeguarding
	set $\safeYY$ does not comprise a multiplier from $\multYY$ provided there is one.
	
	\subsection{Convergence analysis}\label{sec:convergence_rigid_safeguard}
	
	We now present a detailed convergence analysis of the safeguarded ALM from \cref{alg:ALMsafeguarded}.
	To start, let us comment on general properties of the dual sequence $\{(y^k,\hat y^k)\}$.
	
	\begin{mybox}
		\begin{lemma}\label{lem:dual_sequences_safeguarded}
			Suppose that \cref{ass:convexP} holds. 
			Furthermore, let the value function $\mathcal V$ from \eqref{eq:value_function}
			be proper and lsc.
			Let $\{(x^k,z^k,y^k,\hat y^k)\}$ be a sequence generated by 
			\cref{alg:ALMsafeguarded}
			such that \eqref{eq:summability} holds.
			Then 
			\begin{equation}\label{eq:square_summability}
				\sum_{k=0}^\infty\mu_k\norm{y^{k+1}-y^k}^2 < \infty,
				\qquad
				\sum_{k=0}^\infty\mu_k\norm{\hat y^{k+1}-\hat y^k}^2 < \infty
			\end{equation}
			and, particularly, the convergences
			$\mu_k(y^{k+1}-y^k)\to 0$ and $\mu_k(\hat y^{k+1}-\hat y^k)\to 0$ hold.
			Moreover, each accumulation point of $\{y^k\}$ is 
			a solution of the convex optimization problem \eqref{eq:penalized_dual},
			where $\mu\geq 0$ is such that $\mu_k\downto \mu$.
		\end{lemma}
	\end{mybox}
	\begin{proof}
		The proof is closely related to the one of \cref{lem:PPA},
		so we present it briefly.
		
		Again, we define the exact step
		\begin{equation}\label{eq:definition_of_exact_prox}
			y^{k+1}_{\rm e}
			\coloneqq
			\prox_{-\nicefrac1{\mu_k}\dualLL}(\hat y^k)
		\end{equation}
		for each $k\in\N_0$ and obtain the estimate
		\begin{equation}\label{eq:distance_to_exact_prox}
			\norm{y^{k+1}-y^{k+1}_{\rm e}}
			\leq
			\sqrt{2\varepsilon_k/\mu_k}
		\end{equation}
		from \cref{thm:inexactPPA:saddle}, showing $y^{k+1}\in\prox^{\varepsilon_k}_{-\nicefrac{1}{\mu_k}\dualLL}(\hat y^k)$, and \eqref{eq:inexact_prox_mapping}.
		We will make use of the function $\func{\Psi_k}{\YY\times\YY}{\Rinf}$ given by
		\[
		\Psi_k(y,\hat y)
		\coloneqq
		-\dualLL(y)+\indicator_{\safeYY}(\hat y)+\frac{\mu_k}{2}\norm{y-\hat y}^2
		\]
		for each $k\in\N_0$.
		Using inequality \eqref{eq:three_point_inequality}, we find
		\begin{align*}
			-\dualLL(y^{k+1}_{\rm e})
			\leq{}&
			-\dualLL(y^k_{\rm e})
			+
			\frac{\mu_k}{2}\left(
			\norm{\hat y^k-y^k_{\rm e}}^2-\norm{\hat y^k - y^{k+1}_{\rm e}}^2 - \norm{y^{k+1}_{\rm e} - y^k_{\rm e}}^2
			\right),
			\\
			\indicator_{\safeYY}(\hat y^k)
			\leq{}&
			\indicator_{\safeYY}(\hat y^{k-1})
			+
			\frac{\mu_k}{2}\left(
			\norm{y^k-\hat y^{k-1}}^2 - \norm{y^k-\hat y^k}^2-\norm{\hat y^k-\hat y^{k-1}}^2
			\right).
		\end{align*}
		Since $\{\mu_k\}$ is nonincreasing,
		adding up these estimates and using \eqref{eq:four_point_inequality} gives
		\begin{align*}
			\Psi_k(y^{k+1}_{\rm e},\hat y^k)
			\leq{}&
			\Psi_k(y^k_{\rm e},\hat y^{k-1})
			+
			\frac{\mu_k}{2}\left(
			\norm{\hat y^k-y^k_{\rm e}}^2 
			+ 
			\norm{y^k-\hat y^{k-1}}^2 
			- 
			\norm{y^k_{\rm e}-\hat y^{k-1}}^2
			\right.
			\\&\qquad
			\left.-
			\norm{y^k-\hat y^k}^2
			-
			\norm{y^{k+1}_{\rm e}-y^k_{\rm e}}^2
			-
			\norm{\hat y^k-\hat y^{k-1}}^2
			\right)
			\\
			\leq{}&
			\Psi_k(y^k_{\rm e},\hat y^{k-1})
			+
			\frac{\mu_k}{2}\left(
			\norm{y^k-y^k_{\rm e}}^2 - \norm{y^{k+1}_{\rm e}-y^k_{\rm e}}^2
			\right)
			\\
			\leq{}&
			\Psi_{k-1}(y^k_{\rm e},\hat y^{k-1})
			+
			\frac{\mu_k}{2}\left(
			\norm{y^k - y^k_{\rm e}}^2 - \norm{y^{k+1}_{\rm e} - y^k_{\rm e}}^2
			\right).
		\end{align*}
		For some $N\in\N$,
		let us sum up the above inequalities for $k=1,\ldots,N$
		in order to obtain
		\begin{equation*}
			\Psi_{N}(y^{N+1}_{\rm e},\hat y^{N})
			\leq
			\Psi_0(y^1_{\rm e},\hat y^0)
			+
			\sum_{k=1}^{N}
			\frac{\mu_k}{2}\left(
			\norm{y^k - y^k_{\rm e}}^2 - \norm{y^{k+1}_{\rm e} - y^k_{\rm e}}^2
			\right).
		\end{equation*}
		Hence, with the aid of \eqref{eq:distance_to_exact_prox},
		the definition of $\Psi_k$, 
		and \cref{lem:properties_dualLL},
		similar to the arguments  after \eqref{eq:inexact_prox_bound_telescope}, 
		we find
		\begin{equation*}
			\sum_{k=1}^{N}\mu_k\norm{y^{k+1}-y^k}^2
			\leq
			6\sum_{k=0}^\infty \varepsilon_k 
			+
			12\mu_0\sum_{k=0}^\infty \frac{\varepsilon_k}{\mu_k}
			+ 
			6\left( \Psi_0(y^1_{\rm e},\hat y^0)+\Phi^\star \right).
		\end{equation*}
		Taking the limit $N\to\infty$ 
		while having \eqref{eq:summability} in mind, 
		which also yields $\sum_{k=0}^\infty\varepsilon_k<\infty$ and $\sum_{k=0}^\infty \varepsilon_k/\mu_k < \infty$,
		the first inequality in \eqref{eq:square_summability} is obtained,
		and the second one trivially follows as the projection is nonexpansive.
		Then the convergences $\sqrt{\mu_k}(y^{k+1}-y^k)\to 0$ 
		and $\sqrt{\mu_k}(\hat y^{k+1}-\hat y^k)\to 0$
		follow, and boundedness of $\{\mu_k\}$ yields
		$\mu_k(y^{k+1}-y^k)\to 0$ and $\mu_k(\hat y^{k+1}-\hat y^k)\to 0$.
		
		Let us focus now on the second part and
		notice that $\mu$ is well defined (since $\{\mu_k\}$ is bounded and nonincreasing, hence, convergent).
		For each $k\in\N_0$, 
		the definition of the proximal mapping yields
		\begin{equation}\label{eq:approximate_optimality_condition}
			\mu_k(\hat y^{k} - y^{k+1}_{\rm e})
			\in
			\partial(-\dualLL)(y^{k+1}_{\rm e}).
		\end{equation}
		Let $\{y^{k+1}\}_{k\in K}$ be a convergent subsequence with limit $\tilde y\in\YY$. 
		Then $y^{k+1}\to_K\tilde y$, \eqref{eq:distance_to_exact_prox}, and $\varepsilon_k/\mu_k\downto 0$
		yield $y^{k+1}_{\rm e}\to_K\tilde y$.
		Furthermore, $y^{k+1}\to_K\tilde y$
		also implies $\hat y^{k+1}\to_K\proj_{\safeYY}(\tilde y)$ by continuity of the projection
		(which follows from convexity of $\safeYY$).
		Hence, we find $\mu_k\hat y^{k+1}\to_K \mu\proj_{\safeYY}(\tilde y)$ and, thus,
		$\mu_k\hat y^k\to_K\mu\proj_{\safeYY}(\tilde y)$ from the convergence
		$\mu_k(\hat y^{k+1}-\hat y^k)\to 0$.
		Taking the limit $k\to_K\infty$ in \eqref{eq:approximate_optimality_condition}, 
		we end up with
		$\mu(\proj_{\safeYY}(\tilde y)-\tilde y)\in\partial(-\dualLL)(\tilde y)$,
		see \cite[Thm~24.4]{Rockafellar1970}.
		This inclusion means, using
		\cite[Cor.\ 12.30]{BauschkeCombettes2011} again,
		that $\tilde y$ is a stationary point of the convex problem \eqref{eq:penalized_dual},
		i.e., one of its minimizers.
	\end{proof}
	
	Next, we distinguish between the cases 
	where \eqref{eq:P} possesses a stationary minimizer or not.
	Let us start to investigate the regular situation.
	
	\begin{mybox}
		\begin{theorem}\label{thm:global_convergence_safeguarded_alm_dual}
			Suppose that \cref{ass:convexP} holds and \eqref{eq:P} possesses a stationary minimizer.
			Let $\{(x^k,z^k,y^k,\hat y^k)\}$ be a sequence generated by 
			\cref{alg:ALMsafeguarded}
			such that \eqref{eq:summability} holds.
			Furthermore, let $\mu\geq 0$ be such that $\mu_k\downto \mu$.
			Then $\{y^k\}$ is bounded
			and the following hold:
			\begin{enumerate}[label=(\roman*)]
				\item\label{item:penalty_parameter_to_zero} 
				If $\mu=0$, then each accumulation point of $\{y^k\}$
				belongs to $\multYY$.
				\item\label{item:penalty_parameter_uniformly_positive} 
				If $\mu>0$,
				then $\{y^k\}$ converges to a solution of \eqref{eq:penalized_dual}.
				\item\label{item:multiplier_in_box}
				If $\multYY\cap \safeYY\neq\emptyset$,
				then $\{y^k\}$ converges to an element in this set.
			\end{enumerate}
		\end{theorem}
	\end{mybox}
	\begin{proof}
		To start, let us verify that $\{y^k\}$ is bounded.
		Noting that $\{\hat y^k\}\subseteq \safeYY$, $\{\hat y^k\}$ is bounded
		by boundedness of $\safeYY$.
		By assumption $\multYY$ is nonempty, i.e., \eqref{eq:PPA:dual_problem} possesses
		a solution $\bar y$, see \cref{cor:solution_set_dual_problem}. 
		Hence, $\bar y$ is a minimizer of $-\dualLL$,
		and this yields $0\in\partial(-\nicefrac{1}{\mu_k}\dualLL)(\bar y)$, i.e.,
		$\bar y$ is a fixed point of $\prox_{-\nicefrac1{\mu_k}\dualLL}$ for each $k\in\N_0$.
		Then nonexpansiveness of the proximal operator and \eqref{eq:distance_to_exact_prox} guarantee
		\begin{equation}\label{eq:multiplier_sequence_bounded}
			\norm{y^{k+1}-\bar y}
			\leq
			\norm{y^{k+1} - y^{k+1}_{\rm e}}
			+
			\norm{y^{k+1}_{\rm e} - \prox_{-\nicefrac1{\mu_k}\dualLL}(\bar y)}
			\leq
			\sqrt{2\varepsilon_k/\mu_k} + \norm{\hat y^k-\bar y},
		\end{equation}
		where $y^{k+1}_{\rm e}$ is the exact proximal point defined in \eqref{eq:definition_of_exact_prox}. 
		As the right-hand side remains bounded due to \eqref{eq:summability}
		as well as boundedness of $\{\hat y^k\}$,
		$\{y^k\}$ is indeed bounded.
		
		Assertion~\ref{item:penalty_parameter_to_zero} is trivial 
		from \cref{lem:dual_sequences_safeguarded}
		as, for $\mu=0$, the minimizers of \eqref{eq:penalized_dual} are the
		minimizers of $-\dualLL$, i.e., the multipliers from $\multYY$.
		Hence, let us proceed with the proof of 
		assertion~\ref{item:penalty_parameter_uniformly_positive}.
		As $\mu>0$, \cref{alg:ALMsafeguarded} guarantees the existence of $K_0\in\N$
		such that $\mu_k=\mu$ for all $k\in\N$ satisfying $k\geq K_0$.
		Pick a solution $y_\mu\in\YY$ of \eqref{eq:penalized_dual}.
		We note that its solution set is nonempty as $-\dualLL$ is proper and lsc,
		see \cref{lem:value_function,lem:properties_dualLL}, while $\safeYY$ is nonempty and compact.
		Set $\hat{y}_\mu\coloneqq\proj_{\safeYY}(y_\mu)$.
		Then we have $y_\mu=\prox_{-\nicefrac1{\mu}\dualLL}(\hat{y}_\mu)$.
		Particularly, for each $k\in\N$ with $k\geq K_0$,
		it holds
		\begin{align*}
			\norm{y^{k+1}-y_\mu}
			\leq{}&
			\norm{y^{k+1}- y^{k+1}_{\rm e}}+\norm{y^{k+1}_{\rm e} - y_\mu}
			\leq{}
			\sqrt{2\varepsilon_k/\mu_k} 
			+ 
			\norm{\prox_{-\nicefrac1{\mu}\dualLL}(\hat y^k)-\prox_{-\nicefrac1{\mu}\dualLL}(\hat{y}_\mu)}
			\\
			\leq{}&
			\sqrt{2\varepsilon_k/\mu_k} 
			+
			\norm{\hat y^k - \hat{y}_\mu}
			=
			\sqrt{2\varepsilon_k/\mu_k} 
			+
			\norm{\proj_{\safeYY}(y^k) - \proj_{\safeYY}(y_\mu)}
			\leq{}
			\sqrt{2\varepsilon_k/\mu_k} 
			+
			\norm{y^k-y_\mu},
		\end{align*}
		where we used \eqref{eq:distance_to_exact_prox} and
		the nonexpansiveness of the proximal operator.
		Now, similar arguments as in the proof of \cref{lem:PPA}
		(ultimately applying \Cref{lem:Banerts_lemma})
		show that $\{y^k\}$ converges to a minimizer of \eqref{eq:penalized_dual}.
		
		For the proof of assertion~\ref{item:multiplier_in_box}, we can proceed in similar way.
		Obviously, the solutions of \eqref{eq:penalized_dual} 
		correspond to the points in $\multYY\cap \safeYY$
		in this situation, and this set is nonempty by assumption. 
		Picking $y_\mu\in \multYY \cap \safeYY$, we know, in particular, that $y_\mu$ is a minimizer of $-\dualLL$ and, thus, a fixed point
		of $\prox_{-\nicefrac1{\mu_k}\dualLL}$ for each $k\in\N$, 
		and $\proj_{\safeYY}(y_\mu)=y_\mu$ holds as well.
		Hence, for each $k\in\N$, we find the estimates
		\begin{align*}
			\norm{y^{k+1}-y_\mu}
			\leq{}&
			\norm{y^{k+1}- y^{k+1}_{\rm e}}+\norm{y^{k+1}_{\rm e} - y_\mu}
			\leq{}
			\sqrt{2\varepsilon_k/\mu_k} 
			+ 
			\norm{\prox_{-\nicefrac1{\mu_k}\dualLL}(\hat y^k)-\prox_{-\nicefrac1{\mu_k}\dualLL}(y_\mu)}
			\\
			\leq{}&
			\sqrt{2\varepsilon_k/\mu_k} 
			+
			\norm{\hat y^k - y_\mu}
			=
			\sqrt{2\varepsilon_k/\mu_k} 
			+
			\norm{\proj_{\safeYY}(y^k) - \proj_{\safeYY}(y_\mu)}
			\leq{}
			\sqrt{2\varepsilon_k/\mu_k} 
			+
			\norm{y^k-y_\mu}
		\end{align*}
		in similar way as above, and the convergence of $\{y^k\}$ to 
		a point in $\multYY \cap \safeYY$ follows as in the
		proof of \cref{lem:PPA} (via \cref{lem:Banerts_lemma}).
	\end{proof}
	
	\begin{mybox}
		\begin{remark}\label{rem:rigid_safeguard_and_mu_to_zero}
			Let us note that
			in \cref{thm:global_convergence_safeguarded_alm_dual}\,\ref{item:penalty_parameter_to_zero}
			it remains unclear whether the whole sequence $\{y^k\}$ converges.
			Observe that the proof strategy used to verify assertions~\ref{item:penalty_parameter_uniformly_positive}
			and~\ref{item:multiplier_in_box} cannot be applied 
			as we only obtain inequality \eqref{eq:multiplier_sequence_bounded}
			for each $\bar y\in\multYY$ 
			without a chance to estimate $\norm{\hat y^k-\bar y}$
			from above by $\norm{y^k-\bar y}$ as soon as $\bar y\notin\safeYY$
			(the case $\bar y\in\safeYY$ is covered by assertion~\ref{item:multiplier_in_box}).
			Thus, \cref{lem:Banerts_lemma} does not apply
			as we are not in a position to prove the convergence of $\{\norm{y^k-\bar y}\}$.
		\end{remark}
	\end{mybox}
	
	Let us proceed with a statement addressing the primal iterates of \cref{alg:ALMsafeguarded}.
	For this purpose we use the counterpart of estimate \eqref{eq:bound_auglag_optivalue} when \eqref{eq:subproblem_solution} is replaced by \eqref{eq:subproblem_solution:safeguarded}, namely
	\begin{equation}\label{eq:bound_auglag_optivalue:safeguarded}
		\LL_{\mu_k}(x^k,\hat{y}^k)
		\leq
		\inf\LL_{\mu_k}(\cdot,\hat{y}^k)+\varepsilon_k
		\leq
		\inf \Phi + \varepsilon_k
		=
		\Phi^\star + \varepsilon_k ,
	\end{equation}
	easily obtained by substituting $y^k$ with $\hat{y}^k$.
	
	\begin{mybox}
		\begin{corollary}\label{cor:global_convergence_safeguarded_alm_primal}
			Suppose that \cref{ass:convexP} holds and \eqref{eq:P} possesses a stationary minimizer.
			Let $\{(x^k,z^k,y^k,\hat y^k)\}$ be a sequence generated by 
			\cref{alg:ALMsafeguarded}
			such that \eqref{eq:summability} holds.
			Suppose that either $\mu_k\downtoneq 0$ holds or that $\multYY \cap \safeYY\neq\emptyset$ is satisfied.
			Then we have $c(x^k)-z^k\to 0$ and \eqref{eq:upper_limit_optimal}.
			Particularly, each accumulation point $\bar x\in\XX$ of $\{x^k\}$ is a minimizer of \eqref{eq:P}.
			For each subsequence $\{x^k\}_{k\in K}$ such that $x^k\to_K\bar x$,
			we have $z^k\to_K c(\bar x)$ and $f(x^k)+g(z^k)\to_K\Phi^\star$.
		\end{corollary}
	\end{mybox}
	\begin{proof}
		By construction of \cref{alg:ALMsafeguarded}, we have, for all $k\in\N$, that
		\[
		c(x^k)-z^k = \mu_k(y^{k+1}-\hat y^k).
		\]
		Let us argue that the right-hand side tends to zero as $k\to\infty$.
		If $\mu_k\downtoneq 0$, this is trivial from the boundedness
		of $\{y^k\}$ and $\{\hat y^k\}$, see \cref{thm:global_convergence_safeguarded_alm_dual}.
		If $\multYY \cap \safeYY\neq\emptyset$, \cref{thm:global_convergence_safeguarded_alm_dual} 
		yields the existence of $\bar y\in \multYY \cap \safeYY$ such that $y^k\to\bar y$,
		and continuity of the projection guarantees $\hat y^k\to \bar y$ due to $\proj_{\safeYY}(\bar y)=\bar y$.
		Hence, we have $y^{k+1}-\hat y^k\to 0$, and the boundedness of $\{\mu_k\}$ also yields
		$\mu_k(y^{k+1}-\hat y^k)\to 0$.
		Hence, we have shown that $c(x^k)-z^k\to 0$.
		
		Based on \eqref{eq:bound_auglag_optivalue:safeguarded},
		we find
		\begin{multline*}
			f(x^k)+g(z^k)+\innprod{\hat y^k}{c(x^k)-z^k}	
			\leq
			f(x^k)+g(z^k)+\innprod{\hat y^k}{c(x^k)-z^k}+\frac1{2\mu_k}\norm{c(x^k)-z^k}^2
			\\
			=
			\LL_{\mu_k}(x^k,\hat y^k)
			\leq
			\Phi^\star + \varepsilon_k.
		\end{multline*}
		Taking the upper limit $k\to\infty$ while keeping the boundedness of $\{\hat y^k\}$ and the convergences
		$c(x^k)-z^k\to 0$ and $\varepsilon_k\downto 0$ in mind, \eqref{eq:upper_limit_optimal} follows.
		The final statements about accumulation points can be shown similarly as in \cref{thm:global_convergence}.
	\end{proof}
	
	Based on \cref{cor:global_convergence_safeguarded_alm_primal},
	one may ask about primal convergence guarantees of \cref{alg:ALMsafeguarded} in situations where
	the penalty parameter does not tend to zero.
	For instance, this happens in \cref{set:constant_penalty_parameter}.
	As the following simple example illustrates,
	\cref{alg:ALMsafeguarded} may not possess any meaningful
	primal convergence behavior in this situation.
	
	\begin{mybox}
		\begin{example}[Safeguarded ALM with constant penalty parameter]\label{ex:safeguarded_ALM_with_constant_penalty_parameter}
			Consider \eqref{eq:P} with $g\coloneqq\indicator_{\R_+}$ 
			as well as $\func{f}{\R}{\R}$ and $\func{c}{\R}{\R}$ given by
			\[
			f(x)\coloneqq x,\qquad c(x)\coloneqq x.
			\]
			The feasible set of this problem is $\R_+$ and its solution is $\bar x\coloneqq 0$.
			We note that the latter is also stationary with unique multiplier $\bar y\coloneqq -1$.
			
			In \cref{alg:ALMsafeguarded}, we choose $y^0\coloneqq 0$, $\mu_0\coloneqq \mu$ for some $\mu>0$, and $\safeYY\coloneqq\{0\}$.
			Furthermore, we do not update the penalty parameter, i.e.,
			we implement \cref{alg:ALMsafeguarded} in \cref{set:constant_penalty_parameter}.
			Finally, $\varepsilon_k\coloneqq 0$ is used for each $k\in\N_0$.
			Then, in each iteration $k\in\N_0$, we have $\hat y^k = 0$ and $\mu_k = \mu$,
			so that $x^k$ has to be chosen as a minimizer of
			$
			x\mapsto x + \nicefrac{1}{2\mu}\operatorname{min}^2(x,0).
			$
			Hence, we have $x^k=-\mu$ for each $k\in\N_0$,
			and the constant sequence $\{x^k\}$ converges to an infeasible point.

			Notice that convergence continues to fail as long as the set $\safeYY$ does not contain the multiplier $\bar{y}$.
			Take for instance $\safeYY \coloneqq [-\nicefrac{1}{2},\nicefrac{1}{2}]$.
			Then, for any $\mu>0$ and initial $y^0\coloneqq -1$, the primal update $x^k$ which minimizes
			$
			x \mapsto
			x + \nicefrac{1}{2\mu} \operatorname{min}^2(x+\mu \hat{y}^k,0)
			$
			is given by $x^k = -\mu (1 + \hat{y}^k)$, so that the dual update rule gives $y^k = -1$ for all $k\in\N_0$, 
			and, consequently, $\hat{y}^{k}=\proj_{\safeYY}(y^k) = -\nicefrac{1}{2}$.
			Since $x^{k} = -\nicefrac{\mu}{2}$ for each $k\in\N_0$, 
			it is evident that the sequence $\{x^k\}$ does not approach the minimizer $\bar{x}$ but an infeasible point.
		\end{example}
	\end{mybox}
	
	Nevertheless, it is possible to distill satisfying primal convergence properties
	of \cref{alg:ALMsafeguarded} even if the penalty parameter does not tend to zero.
	For example, this happens to be the case in \cref{set:algencan_penalty_parameter},
	see \cite[Thm~4.3]{DeMarchiMehlitz2024} as well.
	
	\begin{mybox}
		\begin{corollary}\label{cor:global_convergence_safeguarded_alm_primal_HPR}
			Suppose that \cref{ass:convexP} holds and \eqref{eq:P} possesses a stationary minimizer. 
			Let $\{(x^k,z^k,y^k,\hat y^k)\}$ be a sequence generated by 
			\cref{alg:ALMsafeguarded}
			such that \eqref{eq:summability} holds.
			Assume that the penalty parameter is updated according to \cref{set:algencan_penalty_parameter}.
			Then we have $c(x^k)-z^k\to 0$ and \eqref{eq:upper_limit_optimal}.
			Particularly, each accumulation point $\bar x\in\XX$ of $\{x^k\}$ is a minimizer of \eqref{eq:P}.
			For each subsequence $\{x^k\}_{k\in K}$ such that $x^k\to_K\bar x$,
			we have $z^k\to_K c(\bar x)$ and $f(x^k)+g(z^k)\to_K\Phi^\star$.
		\end{corollary}
	\end{mybox}
	\begin{proof}
		If $\mu_k\downtoneq 0$, the assertion is clear from \cref{cor:global_convergence_safeguarded_alm_primal}.
		Hence, assume that $\{\mu_k\}$ is bounded away from zero.
		According to \cref{set:algencan_penalty_parameter}, 
		this means that there are $\theta\in(0,1)$ and an index $K_0\in\N$ such that
		\[
		\norm{c(x^k)-z^k}\leq\theta\norm{c(x^{k-1})-z^{k-1}}
		\]
		is valid for all $k\in\N$ such that $k\geq K_0$.
		Hence, $c(x^k)-z^k\to 0$ follows.
		Now, we can proceed as in the proof of \cref{cor:global_convergence_safeguarded_alm_primal}
		to verify the remaining claims.
	\end{proof}
	
	Let us close this section by inspecting the behavior of \cref{alg:ALMsafeguarded}
	in the irregular situation.
	We start with some general observations.
	\begin{mybox}
		\begin{theorem}\label{thm:global_convergence_irregular_safeguarded}
			Suppose that \cref{ass:convexP} holds.
			Furthermore, let all minimizers of \eqref{eq:P} be nonstationary,
			and let the value function $\mathcal V$ from \eqref{eq:value_function}
			be proper and lsc.
			Let $\{(x^k,z^k,y^k,\hat y^k)\}$ be a sequence generated by \cref{alg:ALMsafeguarded}
			such that \eqref{eq:summability} holds.
			Let $\mu\geq 0$ be such that $\mu_k\downto\mu$.
			Then the following assertions hold.
			\begin{enumerate}[label=(\roman{*})]
				\item\label{item:penalty_parameter_to_zero_irregular}  
				If $\mu=0$, 
				then we have $\norm{y^k}\to\infty$ and \eqref{eq:upper_limit_optimal}.
				Particularly, each accumulation point $\bar x\in\XX$ of $\{x^k\}$ 
				is a minimizer of \eqref{eq:P}.
				For each subsequence $\{x^k\}_{k\in K}$ such that $x^k\to_K\bar x$,
				we have $z^k\to_K c(\bar x)$ and $f(x^k)+g(z^k)\to_K\Phi^\star$.
				\item\label{item:penalty_parameter_uniformly_positive_irregular} 
				If $\mu>0$, then $\{y^k\}$ converges to a solution of \eqref{eq:penalized_dual}.
			\end{enumerate}
		\end{theorem}
	\end{mybox}
	\begin{proof}
		Let us start to prove the first assertion.
		Condition $\norm{y^k}\to\infty$ readily follows from \cref{lem:dual_sequences_safeguarded}.
		Indeed, supposing that $\{y^k\}$ possesses a bounded subsequence,
		it has an accumulation point which then must be a solution of \eqref{eq:PPA:dual_problem},
		i.e., a multiplier, see \cref{lem:solutions_of_dual_are_multipliers}, contradicting our assumptions.
		The remaining statements follow precisely as in the proof of \cite[Thm~4.3]{DeMarchiMehlitz2024}.
		The latter result assumes closedness of $\dom g$,
		but this is not required in the present setting as
		$g$ is convex.
		Particularly, in the proof of \cite[Thm~4.3]{DeMarchiMehlitz2024},
		\cite[Lem.\ 2.1]{DeMarchiMehlitz2024} can be replaced by 
		\cite[Prop.\ 4\,(c)]{FriedlanderGoodwinHoheisel2023}
		to remove this assumption.
		
		The second assertion follows from the fact that \eqref{eq:penalized_dual}
		still possesses minimizers by compactness of $\safeYY$ and lower boundedness, properness, and lower semicontinuity of $-\dualLL$,
		see \cref{lem:properties_dualLL}.
		Hence, the proof of \cref{thm:global_convergence_safeguarded_alm_dual}\,\ref{item:penalty_parameter_uniformly_positive}
		applies in the present situation as well.
	\end{proof}
	
	Similar as in \cref{ex:safeguarded_ALM_with_constant_penalty_parameter},
	we can show that \cref{alg:ALMsafeguarded} does not necessarily possess
	any meaningful primal convergence properties in the irregular situation
	if the penalty parameter does not tend to zero.
	
	\begin{mybox}
		\begin{example}[Safeguarded ALM with constant penalty parameter---reloaded]\label{ex:safeguarded_ALM_with_constant_penalty_parameter_irregular}
			Consider \eqref{eq:P} as in \cref{example:cvx_deg}.
			In \cref{alg:ALMsafeguarded}, we choose $y^0\coloneqq 0$, $\mu_0\coloneqq \mu$ for some $\mu>0$, and $\safeYY\coloneqq\{0\}$.
			Furthermore, we 
			implement \cref{alg:ALMsafeguarded} in \cref{set:constant_penalty_parameter}
			and use $\varepsilon_k\coloneqq 0$ for each $k\in\N_0$.
			Then, in each iteration $k\in\N_0$, we have $\hat y^k = 0$ and $\mu_k = \mu$,
			so that $x^k$ has to be chosen as a minimizer of
			$x\mapsto x + \tfrac{1}{2\mu}x^4$.
			Hence, we find $x^k=-\sqrt[3]{\nicefrac{\mu}{2}}$ for each $k\in\N_0$.
			Observe that the constant sequence $\{x^k\}$ converges to an infeasible point.
		\end{example}
	\end{mybox}
	
	If the penalty parameter is updated according to \cref{set:algencan_penalty_parameter}
	in \cref{alg:ALMsafeguarded}, primal convergence guarantees can be given even in situations
	where all minimizers are nonstationary and the penalty parameter does not tend to zero.
	For the purpose of completeness, we state the associated result
	which is a consequence of \cite[Thm~4.3]{DeMarchiMehlitz2024},
	whose assumption that $\dom g$ is closed becomes superfluous in the convex case,
	see \cite[Prop.\ 4\,(c)]{FriedlanderGoodwinHoheisel2023}.
	
	\begin{mybox}
		\begin{proposition}\label{thm:global_convergence_irregular_safeguarded_primal}
			Suppose that \cref{ass:convexP} holds.
			Furthermore, let all minimizers of \eqref{eq:P} be nonstationary,
			and let the value function $\mathcal V$ from \eqref{eq:value_function}
			be proper and lsc.
			Assume that the penalty parameter is updated according to \cref{set:algencan_penalty_parameter}.
			Let $\{(x^k,z^k,y^k,\hat y^k)\}$ be a sequence generated by \cref{alg:ALMsafeguarded}.
			Then we have \eqref{eq:upper_limit_optimal}
			and each accumulation point $\bar x\in\XX$ of $\{x^k\}$ is a minimizer of \eqref{eq:P}.
			For each subsequence $\{x^k\}_{k\in K}$ such that $x^k\to_K\bar x$,
			we have $z^k\to_K c(\bar x)$ and $f(x^k)+g(z^k)\to_K\Phi^\star$.
		\end{proposition}
	\end{mybox}
	
	Note that, in the regular case, we found a slightly different result 
	in \cref{cor:global_convergence_safeguarded_alm_primal_HPR}.
	Therein, the summability condition \eqref{eq:summability} was postulated, additionally,
	and it was possible to show $c(x^k)-z^k\to 0$.
	
	\section{ALM with elastic safeguarding}\label{sec:safeguarded_ALM_elastic}
	
	In \cref{sec:safeguarded_ALM}, 
	we observed that the safeguarded ALM with prescribed safeguard from \cref{alg:ALMsafeguarded} 
	demands a possibly varying penalty parameter to provide some convergence guarantees,
	unless the safeguarding set $\safeYY$ is ``large enough'' to cover a multiplier from $\multYY$ (in case of existence).
	In order to avoid a deliberate choice of $\safeYY$,
	since $\multYY$ is usually unknown a priori,
	and
	to make the AL scheme more adaptive, we now consider
	\cref{alg:ALMsafeguarded_growing},
	where the dual safeguarding set is allowed to grow.
	We say ``elastic'' as opposed to the ``rigid'' strategy with prefixed set $\safeYY$.
	In fact, what really matters in the global convergence analysis \cite{birgin2014practical,DeMarchiMehlitz2024} 
	is that the following condition remains guaranteed:
	\[
	\mu_k \hat{y}^k \to 0
	\quad\text{as}\quad
	\mu_k\downtoneq 0 .
	\]
	Then we can also have an elastic dual safeguarding set $\rho_k \safeYY$, which grows with scaling factor $\rho_k > 0$, 
	as long as $\mu_k\rho_k\downtoneq 0$ whenever $\mu_k\downtoneq 0$.
	This is precisely what happens in \cref{alg:ALMsafeguarded_growing}.
	
	\begin{algorithm2e}[htb]
		\DontPrintSemicolon
		\KwData{$y^0\in \YY$; $\mu_0>0$; nonempty, convex, compact set $\safeYY\subseteq\YY$; 
			$\theta,\eta\in (0,1)$; $\beta\in(\eta^2,\eta)$\;}
		Set $\rho_0 \coloneqq 1$ and $V^{-1}\coloneqq\infty$.\;
		\For{$k = 0,1,2\ldots$}{
			Set $\hat{y}^k\coloneqq \proj_{\rho_k \safeYY}(y^k)$. 
			\label{step:ALMsafeguarded_growing:ysafe}%
			\;
			Select $\varepsilon_k\geq 0$ and compute $x^k\in \XX$ such that \eqref{eq:subproblem_solution:safeguarded} holds.
			\label{step:ALMsafeguarded_growing:subproblem}%
			\;
			Set 
			\label{step:ALMsafeguarded_growing:z}%
			$z^k\coloneqq \prox_{\mu_k g}(c(x^k) + \mu_k \hat y^k)$ and $V^k \coloneqq \norm{c(x^k)-z^k}$.%
			\;
			Set 
			\label{step:ALMsafeguarded_growing:y}%
			$y^{k+1} \coloneqq \hat y^k + \mu_k^{-1} [c(x^k) - z^k]$.%
			\;
			\If{$V^k \leq \theta V^{k-1}$%
				\label{step:ALMsafeguarded_growing:check_V}}{%
				Set $\mu_{k+1} \coloneqq \mu_k$ and $\rho_{k+1} \coloneqq \rho_k$.%
				\label{step:ALMsafeguarded_growing:keep_mu}}
			\Else{%
				Set $\mu_{k+1} \coloneqq \beta \mu_k$ and $\rho_{k+1} \coloneqq \nicefrac{\eta}{\beta}\, \rho_k$.%
				\label{step:ALMsafeguarded_growing:change_mu}}
		}
		\caption{Safeguarded ALM for \eqref{eq:P} with elastic safeguard.}
		\label{alg:ALMsafeguarded_growing}
	\end{algorithm2e}
	
	Observe that the (inverse) stepsizes $\{\mu_k\}$ are merely nonincreasing in the standard ALM from \cref{alg:ALMclassic}
	and its safeguarded version from \cref{alg:ALMsafeguarded},
	whereas we adopt the particular update rule from
	\cref{set:algencan_penalty_parameter}
	in \cref{alg:ALMsafeguarded_growing} to make the elastic mechanism more explicit.
	The particular update at \cref{step:ALMsafeguarded_growing:change_mu} makes it possible to enlarge the safeguarding set, 
	thus going beyond a debatable choice of $\safeYY$,
	while maintaining under control the term $\mu_k \hat{y}^k$.
	In fact, by selecting $\eta \in (\beta,1)$, the scaling factor $\rho_{k+1} > \rho_k$ increases 
	but the product $\mu_{k+1} \rho_{k+1} = \eta \mu_k \rho_k < \mu_k \rho_k$ decreases. 
	For $\eta=\beta$, the safeguarding mechanism in \cref{alg:ALMsafeguarded_growing} coincides with the rigid one in \cref{alg:ALMsafeguarded},
	since $\rho_{k+1}=\rho_k$ for all $k\in\N$.
	Instead, with $\eta\in(0,\beta)$ the scaling factor decreases at \cref{step:ALMsafeguarded_growing:change_mu}, effectively shrinking the safeguard instead of enlarging it.
	Let us also note that choosing $\beta\in(\eta^2,1)$, \cref{step:ALMsafeguarded_growing:change_mu} guarantees
	$\mu_{k+1} \rho_{k+1}^2 = \nicefrac{\eta^2}{\beta}\, \mu_k \rho_k^2 < \mu_k \rho_k^2$, so that we even have
	\begin{equation}\label{eq:elastic_safeguard_requirement_irreg}
		\mu_k \norm{\hat{y}^k}^2 \to 0
		\quad\text{as}\quad
		\mu_k\downtoneq 0 .
	\end{equation}
	The latter will be decisive for the global convergence analysis in the irregular situation.
	
	The elastic safeguard mechanism allows us to unite the advantages of both the classical ALM
	(with strong guarantees in the convex setting) and the rigidly safeguarded ALM
	(which has better global guarantees in the nonconvex setting),
	as revealed by the convergence analysis in \cref{sec:convergence_elastic_safeguard}.
	Although beyond the scope of this work, 
	we note that global convergence results on general composite optimization problems 
	can be readily extended from rigid to elastic safeguarding,
	see \cite[Thm~4.1, Thm~6.2]{birgin2014practical} and \cite[Thm~4.3]{DeMarchiMehlitz2024}.
	
	\begin{mybox}
		\begin{remark}
			As a historical note, we shall mention 
			Robinson's doctoral dissertation and the primal-dual BCL algorithm therein, see \cite[Alg.\ 4.2.1]{robinson2007primal},
			where expanding ``artificial bounds'' are imposed on the dual variables.
			This mechanism was then interpreted in the same context as a ``regularization by bounding the multipliers''
			by Gill and Robinson \cite[\S 4]{GillRobinson2012}.
			However, their augmented Lagrangian scheme (or their convergence analysis, at least)
			relies on the classical BCL approach of Conn et al.\ \cite{conn1991globally}
			as globalization technique,
			and not on safeguarding.
			
			To the best of our knowledge,
			this idea was reconsidered only in \cite[Alg.\ 3]{demarchi2021augmented},
			where it is truly adopted as a safeguard for globalization purposes,
			in the sense of \cite{andreani2008augmented,birgin2014practical,kanzow2017example},
			but theoretically analyzed only for the simple case 
			where the scaling factors $\{\rho_k\}$ remain bounded.
		\end{remark}
	\end{mybox}
	
	\subsection{Convergence analysis}\label{sec:convergence_elastic_safeguard}
	
	To study the behavior of the dual sequence $\{y^k\}$ generated by \cref{alg:ALMsafeguarded_growing}, 
	in analogy with \eqref{eq:alternating_operators},
	we have to investigate the more dynamic dual update scheme
	\begin{equation*}
		y^{k+1}
		\in
		\prox_{- \nicefrac{1}{\mu_k} \dualLL}^{\varepsilon_k} \left( \proj_{\rho_k \safeYY}(y^k) \right)
	\end{equation*}
	which is not covered by \cite{ArizaLopezNicolae2015,banert2014backward};
	not only is the proximal operator evaluated
	inexactly, but also the inner operator changes along the iterations.
	
	The following \cref{thm:global_convergence_special_safeguarded_alm} relates to \cref{alg:ALMsafeguarded_growing} and enhances the dual convergence results for \cref{alg:ALMsafeguarded},
	obtained in \cref{thm:global_convergence_safeguarded_alm_dual}, in two ways.
	First, even for $\mu_k\downtoneq 0$, the full convergence of the dual sequence $\{y^k\}$ is obtained.
	Second, even for $\mu_k\downto\mu$ for some $\mu>0$, the limit is a multiplier,
	and not merely a minimizer of \eqref{eq:penalized_dual}.
	
	\begin{mybox}
		\begin{theorem}\label{thm:global_convergence_special_safeguarded_alm}
			Suppose that \cref{ass:convexP} holds and \eqref{eq:P} possesses a stationary minimizer.
			Let $\{(x^k,z^k,y^k,\hat y^k)\}$ be a sequence generated by 
			\cref{alg:ALMsafeguarded_growing}
			such that \eqref{eq:summability} holds.
			Then $\{y^k\}$ converges to a multiplier in $\multYY$.
			Furthermore, we have $c(x^k)-z^k\to 0$ and \eqref{eq:upper_limit_optimal}.
			Particularly, each accumulation point $\bar x\in\XX$ of $\{x^k\}$ is a minimizer of \eqref{eq:P}.
			For each subsequence $\{x^k\}_{k\in K}$ such that $x^k\to_K\bar x$,
			we have $z^k\to_K c(\bar x)$ and $f(x^k)+g(z^k)\to_K\Phi^\star$.
		\end{theorem}
	\end{mybox}
	\begin{proof}
		As in the proof of \cref{lem:dual_sequences_safeguarded},
		we make use of the exact step $y^{k+1}_{\rm e}$ from \eqref{eq:definition_of_exact_prox}
		for each $k\in\N_0$.
		Following the reasoning in in the proof of \cref{lem:dual_sequences_safeguarded},
		for each $k\in\N_0$, $y^{k+1}\in\prox_{-\nicefrac1{\mu_k}\dualLL}^{\varepsilon_k}(\hat y^k)$
		guarantees validity of
		\eqref{eq:distance_to_exact_prox} and
		\eqref{eq:approximate_optimality_condition}.
		We distinguish two cases.
		
		$\triangleright$ Assume that $\mu_k\downto \mu$ holds for some $\mu>0$.
		Then, by construction of \cref{alg:ALMsafeguarded_growing} (and specifically the adoption of \cref{set:algencan_penalty_parameter}), 
		there are a finite $\rho>0$ and some $K_0\in\N$ such that $\mu_k=\mu$ and $\rho_k=\rho$ hold for all $k\geq K_0$.
		Furthermore, as $\{\mu_k\}$ is bounded away from zero, we find $c(x^k)-z^k\to 0$,
		which implies $y^{k+1}-\hat y^k\to 0$.
		\\
		Let $\{y^{k+1}\}_{k\in K}$ be a convergent subsequence of $\{y^k\}$ (which we show to exist below) with limit $\tilde y\in\YY$.
		Then we also have $y^{k+1}_{\rm e}\to_K\tilde y$ from \eqref{eq:distance_to_exact_prox}, 
		and $y^{k+1}_{\rm e}-y^{k+1}\to 0$ follows by $\mu_k=\mu>0$ for all $k\geq K_0$.
		Hence, $y^{k+1}_{\rm e}-\hat y^k\to 0$ holds,
		and taking the limit $k\to_K\infty$ in \eqref{eq:approximate_optimality_condition}
		yields $0\in\partial(-\dualLL)(\tilde y)$, see \cite[Thm~24.4]{Rockafellar1970} again,
		i.e., $\tilde y$ is a stationary point of $-\dualLL$ and, thus, a solution of \eqref{eq:PPA:dual_problem}.
		Particularly, $\tilde y\in\multYY$ follows by \cref{cor:solution_set_dual_problem}.
		\\
		On the one hand, we note that $\{\hat y^k\}_{k\in K}\subseteq\rho \safeYY$ holds.
		On the other hand, $y^{k+1}-\hat y^k\to 0$ and $y^{k+1}\to_K\tilde y$
		yield $\hat y^k\to_K\tilde y$.
		Closedness of $\rho \safeYY$ guarantees $\tilde y\in\rho \safeYY$.
		Hence, we have shown that each accumulation point of $\{y^k\}$ belongs to $\multYY\cap(\rho \safeYY)$.
		\\
		Let us argue that $\{y^k\}$ is bounded in the present situation,
		which yields the existence of a convergent subsequence and, thus,
		that $\multYY\cap(\rho \safeYY)$ is nonempty.
		Indeed, as we have $\{\hat y^k\}\subseteq \rho \safeYY$, $\{\hat y^k\}$ is bounded.
		Recalling $y^{k+1}-\hat y^k\to 0$,
		$\{y^k\}$ must also be bounded.
		\\
		Pick any $k\in\N$ such that $k\geq K_0$.
		For each $\bar{y}_\rho\in \multYY\cap(\rho \safeYY)$, we have
		\begin{align*}
			\norm{y^{k+1}-\bar{y}_\rho}
			\leq{}&
			\norm{y^{k+1}-y^{k+1}_{\rm e}} + \norm{y^{k+1}_{\rm e}-\bar{y}_\rho}
			\leq{}
			\sqrt{2\varepsilon_k/\mu} + \norm{\prox_{-\nicefrac1{\mu}\dualLL}(\hat y^k)-\prox_{-\nicefrac1{\mu}\dualLL}(\bar{y}_\rho)}
			\\
			\leq{}&
			\sqrt{2\varepsilon_k/\mu} + \norm{\hat y^k-\bar{y}_\rho}
			=
			\sqrt{2\varepsilon_k/\mu} + \norm{\proj_{\rho \safeYY}(y^k)-\proj_{\rho \safeYY}(\bar{y}_\rho)}
			\leq{}
			\sqrt{2\varepsilon_k/\mu} + \norm{y^k-\bar{y}_\rho}
		\end{align*}
		by \eqref{eq:distance_to_exact_prox} and nonexpansiveness of the proximal operator.
		As in the proof of \cref{lem:PPA},
		we combine this estimate with the inexactness requirement \eqref{eq:summability} in order to find 
		the convergence of $\{y^k\}$ to a point in $\multYY\cap(\rho \safeYY)$ through \cref{lem:Banerts_lemma}.
		\\
		Finally, retracing the proof of \cref{cor:global_convergence_safeguarded_alm_primal},
		\eqref{eq:upper_limit_optimal} and the remaining claims about primal accumulation points 
		follow from the boundedness of $\{\hat y^k\}$ that we mentioned earlier.
		
		$\triangleright$ Next, we consider the situation where $\mu_k\downtoneq 0$.
		By construction of \cref{alg:ALMsafeguarded_growing}, we have $\rho_k\to\infty$ and $\mu_k\hat y^k\to 0$.
		Choosing a convergent subsequence $\{y^{k+1}\}_{k\in K}$ with limit $\tilde y\in\YY$,
		we have $y^{k+1}_{\rm e}\to_K\tilde y$ from \eqref{eq:distance_to_exact_prox},
		and taking the limit $k\to_K\infty$ in \eqref{eq:approximate_optimality_condition}
		while respecting \cite[Thm~24.4]{Rockafellar1970}
		yields $0\in\partial(-\dualLL)(\tilde y)$ as $\{y^{k+1}_{\rm e}\}_{k\in K}$ is bounded.
		Hence, $\tilde y$ is a solution of \eqref{eq:PPA:dual_problem} and, thus, belongs to $\multYY$,
		see \cref{cor:solution_set_dual_problem} again.
		\\
		Pick $\bar{y} \in \multYY$ arbitrarily.
		As we have $\rho_k\to\infty$, $\bar{y}\in \rho_k \safeYY$ holds for sufficiently large $k\in\N$.
		For any such $k\in\N$, by \eqref{eq:distance_to_exact_prox} we find
		\begin{align*}
			\norm{y^{k+1}-\bar{y}}
			\leq{}&
			\norm{y^{k+1}-y^{k+1}_{\rm e}} + \norm{y^{k+1}_{\rm e}-\bar{y}}
			\leq{}
			\sqrt{2\varepsilon_k/\mu_k} + \norm{\prox_{-\nicefrac1{\mu_k}\dualLL}(\hat y^k)-\prox_{-\nicefrac1{\mu_k}\dualLL}(\bar{y})}
			\\
			\leq{}&
			\sqrt{2\varepsilon_k/\mu_k} + \norm{\hat y^k-\bar{y}}
			=
			\sqrt{2\varepsilon_k/\mu_k} + \norm{\proj_{\rho_k \safeYY}(y^k)-\proj_{\rho_k \safeYY}(\bar{y})}
			\leq{}
			\sqrt{2\varepsilon_k/\mu_k} + \norm{y^k-\bar{y}},
		\end{align*}
		and the convergence of $\{y^k\}$ to a point in $\multYY$ follows as above, relying on \cref{lem:PPA,lem:Banerts_lemma}.
		\\
		Finally, let $\bar y\in \multYY$ be the limit of $\{y^k\}$.
		As $\{y^k\}$ is bounded while $\rho_k\to\infty$ holds,
		we know that $\hat y^k=y^k$ holds for all sufficiently large $k\in\N$.
		Hence, we find $\hat y^k\to \bar y$ which guarantees that $\{\hat y^k\}$ is bounded.
		Consequently, we have $\mu_k(y^{k+1}-\hat y^k)\to 0$.
		By construction of \cref{alg:ALMsafeguarded_growing}, this gives $c(x^k)-z^k\to 0$.
		Finally, \eqref{eq:upper_limit_optimal} and the remaining claims about
		primal accumulation points follow as in the proof of
		\cref{cor:global_convergence_safeguarded_alm_primal} since $\{\hat y^k\}$ is bounded.
	\end{proof}
	
	Let us emphasize that property \eqref{eq:elastic_safeguard_requirement_irreg}
	has not been used in the proof of \cref{thm:global_convergence_special_safeguarded_alm},
	i.e., it is not necessary to claim $\eta^2<\beta$ in \cref{alg:ALMsafeguarded_growing}
	if problem \eqref{eq:P} is known to possess stationary minimizers.

	To close the section, let us inspect the irregular case.
	The following is a counterpart of \cref{thm:global_convergence_irregular_safeguarded} 
	for the special \cref{alg:ALMsafeguarded_growing}, which implements \cref{set:algencan_penalty_parameter}.
	Thanks to the elastic safeguarding in \cref{alg:ALMsafeguarded_growing},
	the two cases appearing in \cref{thm:global_convergence_irregular_safeguarded} can be merged.
	
	\begin{mybox}
		\begin{theorem}\label{thm:global_convergence_special_safeguarded_alm_irregular}
			Suppose that \cref{ass:convexP} holds.
			Furthermore, let all minimizers of \eqref{eq:P} be nonstationary,
			and let the value function $\mathcal V$ from \eqref{eq:value_function}
			be proper and lsc.
			Let $\{(x^k,z^k,y^k,\hat y^k)\}$ be a sequence generated by 
			\cref{alg:ALMsafeguarded_growing}
			such that \eqref{eq:summability} holds.
			Then we have $\norm{y^k}\to\infty$ and \eqref{eq:upper_limit_optimal}.
			Particularly, each accumulation point $\bar x\in\XX$ of $\{x^k\}$ is a minimizer of \eqref{eq:P}.
			For each subsequence $\{x^k\}_{k\in K}$ such that $x^k\to_K\bar x$,
			we have $z^k\to_K c(\bar x)$ and $f(x^k)+g(z^k)\to_K\Phi^\star$.
		\end{theorem}
	\end{mybox}
	\begin{proof}
		From the proof of \cref{thm:global_convergence_special_safeguarded_alm}
		we observe that any accumulation point of $\{y^k\}$ is a minimizer of
		$-\dualLL$ and, thus, an element of $\multYY$, see \cref{lem:solutions_of_dual_are_multipliers}, 
		which is assumed to be empty.
		Hence, $\{y^k\}$ does not possess a convergent subsequence, and $\norm{y^k}\to\infty$ follows.
		We proceed by distinguishing two cases.
		
		$\triangleright$ Assume that $\mu_k\downarrow 0$.
		Relation \eqref{eq:bound_auglag_optivalue:safeguarded} gives
		\begin{multline*}
			f(x^k)+g(z^k)+\frac{\mu_k}{2}\left( \norm{y^{k+1}}^2-\norm{\hat{y}^k}^2 \right)
			=
			f(x^k)+g(z^k)+\frac{1}{2\mu_k}\norm{c(x^k)+\mu_k \hat{y}^k-z^k}^2-\frac{\mu_k}{2}\norm{\hat{y}^k}^2
			\\
			=
			\LL_{\mu_k}(x^k,\hat{y}^k)
			\leq
			\Phi^\star+\varepsilon_k.
		\end{multline*}
		From \eqref{eq:elastic_safeguard_requirement_irreg} we know that $\mu_k\norm{\hat y^k}^2 \to 0$,
		and the latter yields
		\[
		\liminf_{k\to\infty}\frac{\mu_k}{2}\left( \norm{y^{k+1}}^2-\norm{\hat{y}^k}^2 \right)\geq 0.
		\]
		Hence, taking the upper limit $k\to\infty$ in the above estimate, we find \eqref{eq:upper_limit_optimal}.
		
		Let us pick an arbitrary accumulation point $\bar{x}$ of $\{x^k\}$,
		and let $K\subseteq\N$ be such that $x^k\to_K\bar x$.
		To show the final claims,
		one can proceed as in the proof of \cref{thm:global_convergence}
		provided that we have $c(x^k)-z^k\to_K 0$.
		First, we exploit $\mu_k\hat y^k\to 0$, which holds by construction of \cref{alg:ALMsafeguarded_growing}, 
		and \cite[Prop.\ 4\,(c)]{FriedlanderGoodwinHoheisel2023}
		to find
		\[
		\mu_k g(z^k) + \frac12\norm{z^k-c(x^k)-\mu_k\hat y^k}^2 \to_K 0.
		\]
		Second, applying \cite[Lem.\ 2.2]{DeMarchiMehlitz2024} 
		yields $\norm{z^k-c(x^k)-\mu_k\hat y^k}\to_K 0$.
		Using $\mu_k\hat y^k\to 0$ once again, $c(x^k)-z^k\to_K 0$ follows.
		
		$\triangleright$ Let us assume that $\mu_k\downto \mu$ for some $\mu>0$.
		To start, we note that, due to the update mechanism at \cref{step:ALMsafeguarded_growing:change_mu,step:ALMsafeguarded_growing:check_V} 
		of \cref{alg:ALMsafeguarded_growing}, $V^k\downto 0$ follows
		which yields $c(x^k)-z^k\to 0$ and that $\{\rho_k\}$ and, thus, $\{\hat y^k\}$ remain bounded.
		Hence, validity of \eqref{eq:upper_limit_optimal} can be shown as in \cref{cor:global_convergence_safeguarded_alm_primal},
		and the final claims follow as in \cref{thm:global_convergence}.
	\end{proof}
	
	Let us note that, in the setting of \cref{thm:global_convergence_special_safeguarded_alm_irregular},
	we even have $c(x^k)-z^k\to 0$ if $\{\mu_k\}$ remains bounded away from zero (as mentioned in the proof).
	To obtain $c(x^k)-z^k\to 0$ also in the case $\mu_k\downtoneq 0$,
	due to $\mu_k\hat y^k\to 0$, it would be necessary to verify $\mu_k y^{k+1}\to 0$,
	but we are not aware of a proof or a counterexample for this property.
	Note, however, that in 
	\cref{thm:global_convergence_irregular_safeguarded,thm:global_convergence_irregular_safeguarded_primal},
	we also do not obtain the full convergence of $\{c(x^k)-z^k\}$ in general.
	Apart from this minor drawback, in the considered irregular case,
	\cref{thm:global_convergence_special_safeguarded_alm_irregular} 
	provides significantly better global convergence properties for \cref{alg:ALMsafeguarded_growing}
	than \cref{thm:global_convergence_irregular} does for \cref{alg:ALMclassic}.
	
	\section{Illustrative examples}\label{sec:summary_and_numerics}
	
	In this section, we illustrate our findings by means of three example problems,
	comparing different safeguarding techniques and update rules for the penalty parameter.
	
	\paragraph*{Problems}
	Three one-dimensional toy problems are considered to illustrate the behavior of solvers under different circumstances.
	All of them possess a single inequality constraint and, thus, can be modeled as a convex composite problem
	of type \eqref{eq:P} such that the indicator function of the nonpositive real line plays the role of $g$.
	\begin{itemize}
		\item
		\emph{Regular example}:
		A convex problem with the unique, stationary minimizer $\bar{x} \coloneqq 0$
		and associated multiplier $\bar{y} \coloneqq 1$:
		\begin{equation}\label{problem:cvx_reg}
			\minimize_{x} ~ x \quad\stt\quad x^2-x \leq 0 .
		\end{equation}
		\item
		\emph{Irregular example}:
		A convex problem with the unique, nonstationary minimizer $\bar{x} \coloneqq 0$, as observed in \cref{example:cvx_deg}:
		\begin{equation}\label{problem:cvx_deg}
			\minimize_{x} ~ x \quad\stt\quad x^2\leq 0 .
		\end{equation}
		\item
		\emph{Kanzow--Steck example}:
		A nonconvex formulation of a convex problem with the unique, stationary minimizer $\bar{x} \coloneqq 1$
		and associated multiplier $\bar{y} \coloneqq \nicefrac{1}{3}$,
		see \cref{rem:ks_example_nonconvex}:
		\begin{equation}\label{problem:ks_reg}
			\minimize_{x} ~ x \quad\stt\quad 1-x^3\leq 0 .
		\end{equation}
		Although \eqref{problem:ks_reg} does not comply with \cref{ass:convexP} due to its nonconvex formulation,
		we can inspect how different ALMs behave for this example.
	\end{itemize}
	
	\paragraph*{Setup}
	We execute \cref{alg:ALMclassic,alg:ALMsafeguarded,alg:ALMsafeguarded_growing}
	in different configurations.
	First, the penalty update follows either \cref{set:constant_penalty_parameter} (\emph{fixed}) or
	\cref{set:algencan_penalty_parameter} (\emph{adaptive}) with $\beta = \nicefrac{1}{2}$ and $\theta = \nicefrac{9}{10}$.
	The initial penalty value is always $\mu_0 = 1$.
	Second, the safeguarding mechanism can be omitted (\emph{none}), \emph{rigid}, or \emph{elastic}.
	In case of \emph{elastic} safeguarding, we only consider the \emph{adaptive} update mechanism
	for the penalty parameter as studied in \cref{sec:safeguarded_ALM_elastic}.
	When required, we set $\safeYY = [-y_{\max},y_{\max}]$ with $y_{\max}\coloneqq \nicefrac{1}{10}$,
	strictly smaller than the optimal multiplier for our examples,
	and $\eta = \nicefrac{3}{5}$ for the increase of $\rho$.
	The inner stopping condition \eqref{eq:subproblem_solution} is replaced with $\| \nabla_x \LL_{\mu_k}(x^k,y^k) \| \leq \varepsilon_k$ (and \eqref{eq:subproblem_solution:safeguarded} with $\|\nabla_x \LL_{\mu_k}(x^k,\hat{y}^k) \| \leq \varepsilon_k$ for the safeguarded variants).
	This condition is not only simpler to check and more practical for the nonconvex example, but it allows us also to monitor stationarity of the primal-dual pair $(x^k,y^k)$.
	The sequence of inner tolerances is generated according to
	$\varepsilon_{k+1} = \max\{\epsilon,\kappa_\varepsilon\varepsilon_k\}$ 
	with $\kappa_\varepsilon=\nicefrac{1}{2}$ and $\varepsilon_0 = 1$.
	The primal-dual tolerance $\epsilon=10^{-9}$ defines also the termination conditions:
	the algorithm is aborted when $\varepsilon_k\leq \epsilon$ and $|c(x^k)-z^k| \leq \epsilon$.
	All runs are initialized with $x^0=0$ and $y^0=0$
	and all inner subproblems to compute $x^k$ are warm-started at the previous point $x^{k-1}$
	(relevant only for the nonconvex example \eqref{problem:ks_reg}).

	The results obtained are illustrated by monitoring and reporting three quantities:
	the primal residual $V^k \coloneqq | c(x^k)-z^k |$,
	the penalty parameter $\mu_k$,
	and the multiplier norm $|y^k|$
	against the iteration number $k$.
	The dual residual is omitted since it is directly controlled by $\varepsilon_k$ 
	at \cref{step:ALMsafeguarded_growing:subproblem}.
	We adopt a persistent pattern to depict the ALM configurations in the figures below:
	safeguards can be none (\textcolor{color_none}{yellow line}), rigid (\textcolor{color_rigid}{red line}), and elastic (\textcolor{color_elastic}{blue line});
	penalty parameters can be fixed (\underdashed{dashed line}) and adaptive (\underline{solid line}).
	
	\paragraph*{Discussion}
	The numerical results for the \emph{regular} example \eqref{problem:cvx_reg} are depicted in \cref{fig:cvx_reg} and confirm the theoretical observations 
	in \cref{thm:global_convergence_special_safeguarded_alm,thm:global_convergence,thm:global_convergence_safeguarded_alm_dual}.
	Particularly, notice that 
	all variants converge to the solution $(\bar{x},\bar{y})$
	except the one with fixed penalty parameter and rigid safeguard, 
	because the selected safeguarding set $\safeYY$ is too small. 
	Convergence can be rescued by decreasing the penalty parameter (adaptive penalty)
	and/or by enlarging $\safeYY$ with $\rho \safeYY$ (elastic safeguard).

	Analogous observations are valid for the \emph{irregular} example \eqref{problem:cvx_deg}, reported in \cref{fig:cvx_deg}.
	All variants converge to the solution $\bar{x}$
	except the one with fixed penalty parameter and rigid safeguard.

	Finally, let us consider the \emph{Kanzow--Steck} example \eqref{problem:ks_reg}.
	Despite the strong similarity with \eqref{problem:cvx_reg},
	the convergence results in \cref{fig:cvx_reg,fig:ks_reg} are qualitatively different
	for the variants with fixed penalty parameter.
	The discrepancy is explained by the nonconvexity afflicting \eqref{problem:ks_reg},
	which introduces spurious local minima in the subproblems.
	When warm-starting the inner solver at $x^{k-1}$, variants with fixed penalty parameter generate iterates that remain negative, preventing to find feasible points.
	Cold-starting at positive values unlocks convergence, and the assessment made for \cref{fig:cvx_reg} remains valid.
	In particular, the variant with elastic safeguard converges, faster (and with less penalty updates) than with the rigid one.

	\begin{figure}[tbh]
		\centering%
		\includegraphics{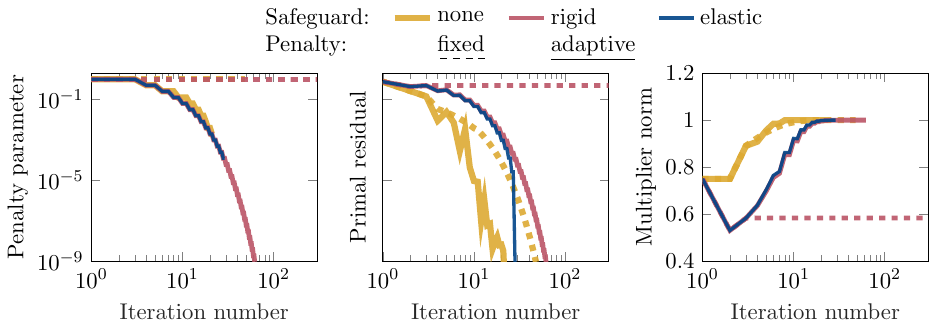}%
		\caption{%
			Regular example: Convex problem with stationary solution \eqref{problem:cvx_reg}.
			All variants converge to the primal-dual solution $(\bar{x},\bar{y})$
			except the one with fixed safeguard and fixed penalty parameter, 
			because the selected safeguarding set $\safeYY$ is too small (middle and right panels).
			The \emph{elastic} variant displays two regimes: 
			first, it behaves as with rigid safeguard when parameter $\rho_k$ is still small, 
			then it switches to a faster convergence 
			as for the classical scheme without safeguard (middle panel). 
			Such transition mitigates the number of penalty parameter updates (left panel).
		}%
		\label{fig:cvx_reg}%
	\end{figure}
	
	\begin{figure}[tbh]
		\centering%
		\includegraphics{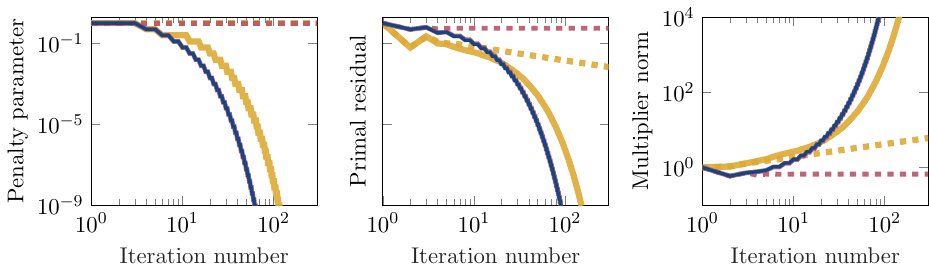}%
		\caption{%
			Irregular example: Convex problem without stationary solution \eqref{problem:cvx_deg}. 
			Legend as in \cref{fig:cvx_reg}.
			All variants converge to the solution $\bar{x}$
			except the one with fixed penalty parameter and rigid safeguard (middle panel).
			All variants with adaptive penalty quickly converge to the solution,
			with $\mu_k\downtoneq 0$ (left panel).
			The classical variant with fixed penalty parameter slowly converges, thanks to $|y^k|\to\infty$ (right panel).
		}%
		\label{fig:cvx_deg}%
	\end{figure}
	
	\begin{figure}[tbh]
		\centering%
		\includegraphics{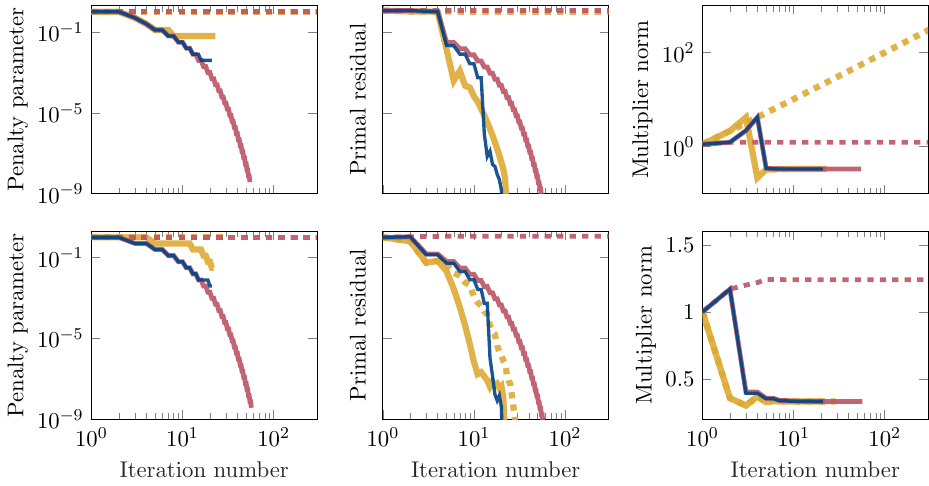}%
		\caption{%
			Kanzow--Steck example: 
			Nonconvex formulation of convex problem with stationary solution \eqref{problem:ks_reg}.
			Legend as in \cref{fig:cvx_reg}.
			The nonconvex subproblem is warm-started at $x^{k-1}$ (top panels) or cold-started at $1$ (bottom panels).
			When warm-started, all variants with adaptive penalty converge to the solution, and
			only the one with rigid safeguard exhibits $\mu_k\downtoneq 0$,
			while variants with fixed penalty parameter fail to converge,
			as they generate primal iterates $x^k$ approaching zero from below.
			When cold-started, the behaviors observed on the \emph{regular} example are resumed.
		}%
		\label{fig:ks_reg}%
	\end{figure}

	\section{Final remarks and open questions}
	\label{sec:conclusions}
	
	In this paper, we illustrated that the correspondence between ALM and PPA, 
	well known for inequality-constrained convex problems, 
	extends to fully convex composite optimization problems.
	The latter has been used to conduct a global convergence analysis for
	the classical ALM, a rigidly safeguarded ALM, and a novel elastically safeguarded ALM
	both in the regular setting, where minimizers are assumed to be stationary,
	and in the irregular situation, where all minimizers are nonstationary.
	The ALM with elastic safeguard turned out to consolidate the advantages of both the classical and rigidly
	safeguarded version while effectively extenuating their respective disadvantages.
	It might be interesting to explore whether similar results can be obtained
	for (safeguarded) proximal ALMs, see \cite{rockafellar2024generalizations}.
	
	In our future research, we aim to illustrate the efficiency of elastic safeguarding
	for general composite optimization problems. As mentioned in \cref{sec:safeguarded_ALM_elastic},
	global convergence properties of associated ALMs should be easy to obtain simply
	by paralleling results and proofs available for rigidly safeguarded ALMs.
	However, we claim that elastic safeguarding may also be beneficial for a local analysis
	(guaranteeing local fast convergence to a primal-dual pair).
	Here, for ALMs with rigid safeguard, one typically has to assume (or ensure by comparatively
	hard assumptions) that the safeguarding becomes unnecessary for large enough iterations,
	i.e., that the dual variable at the limit is in the safeguarding set,
	see, e.g., \cite[\S~4.3]{DeMarchiMehlitz2024} and \cite[\S~4.3]{steck2018dissertation}. 
	In contrast, it is likely that such an assumption is not required
	when elastic safeguarding is used.
	
\phantomsection
\addcontentsline{toc}{section}{References}%
{\small\bibliographystyle{habbrv}
	\bibliography{biblio}}

\appendix
\section{Recession functions of Moreau envelopes}
\label{sec:appendix}
	
Here, we are going to present a proof of \cref{lem:recession_function_of_moreau_envelope}.
Therefore, we have to introduce some additional notation.

For proper, lsc, convex functions $\func{\varphi}{\YY}{\Rinf}$ and $\func{\psi}{\YY}{\Rinf}$,
we refer to $\func{\varphi\square\psi}{\YY}{\Rinf\cup\{\pm\infty\}}$ given by
\begin{equation*}
	(\varphi\square\psi)(z)
	\coloneqq
	\inf_w\{\varphi(w)+\psi(z-w)\,|\,w\in\YY\}
\end{equation*}
as the \emphdef{infimal convolution} of $\varphi$ and $\psi$.
Observe that $\varphi^\gamma = \varphi \square \nicefrac{1}{2\gamma}\norm{\cdot}^2$ holds for all $\gamma>0$.

For each $t>0$, the proper, lsc, convex function $t\epimult \varphi\colon\YY\to\Rinf$ is given by
\[
(t\epimult \varphi)(z)\coloneqq t\,\varphi(z/t),
\]
and the operation $\epimult$ is referred to as \emphdef{epi-multiplication} in the literature
due to $\epi (t\epimult\varphi)=t\,\epi\varphi$.

A sequence of functions $\{\varphi_k\}$ such that $\func{\varphi_k}{\YY}{\Rinf}$ for each $k\in\N$ is said to be \emphdef{epi-convergent} to
$\varphi$ if, for each $y\in\YY$, there exists a sequence $\{\bar y_k\}\subseteq \YY$ converging to $y$ such that
$
\limsup_{k\to\infty} \varphi_k(\bar y_k)\leq \varphi(y)
$,
and for each sequence $\{y_k\}\subseteq\YY$ converging to $y$,
$
\liminf_{k\to\infty} \varphi_k(y_k)\geq \varphi(y)
$
is guaranteed. The epi-convergence of $\{\varphi_k\}$ to $\varphi$ will be denoted by $\varphi_k\epito \varphi$.
In the case where $\{\varphi_t\}_{t>0}$ is a family of functions such that $\func{\varphi_t}{\YY}{\Rinf}$ for each $t>0$,
we say that $\{\varphi_t\}_{t>0}$ epi-converges to $\varphi$ as $t\downtoneq 0$ if, for each null sequence $\{t_k\}\subset(0,\infty)$, 
$\varphi_{k}\epito\varphi$ holds for $\{\varphi_k\}$ given by $\varphi_k\coloneqq \varphi_{t_k}$ for each $k\in\N$,
and we use $\varphi_t\epito\varphi$ to denote this property.

\begin{proof}[Proof of \cref{lem:recession_function_of_moreau_envelope}]
	Throughout the proof, let $\gamma>0$ be arbitrarily chosen.
	For each $t>0$ and $z\in\YY$, we find
	\begin{multline*}
		(t\epimult \varphi^\gamma)(z)
		={}
		t\,\inf_{w}\left\{ \varphi(w)+\frac{1}{2\gamma}\norm{w-z/t}^2 \right\}
		={}
		\inf_{w}\left\{ t\,\varphi(tw/t)+\frac{1}{2t\gamma}\norm{tw-z}^2 \right\}
		\\
		={}
		\inf_{w'}\left\{ (t\epimult \varphi)(w')+\frac{1}{2t\gamma}\norm{w'-z}^2 \right\}
		={}
		\left( (t\epimult\varphi)\square \frac{1}{2t\gamma}\norm{\cdot}^2 \right)(z).
	\end{multline*}
	Due to \cite[Lem.\ 1]{FriedlanderGoodwinHoheisel2023}, we have $t\epimult\varphi\epito\varphi_\infty$
	as $t\downtoneq 0$.
	Furthermore, we can exploit $\nicefrac{1}{2t\gamma}\norm{\cdot}^2=t\epimult \nicefrac{1}{2\gamma}\norm{\cdot}^2$ and
	\cite[Lem.\ 3.5]{BurkeHoheisel2017} in order to find $\nicefrac{1}{2t\gamma}\norm{\cdot}^2\epito\indicator_{\{0\}}$
	as $t\downtoneq 0$.
	Noting that we have $K_\infty=K$ for each closed, convex cone $K\subseteq\YY$,
	\begin{equation*}
		(\epi \varphi_\infty)_\infty \cap \left( -(\epi\indicator_{\{0\}})_\infty \right)
		=
		((\epi \varphi)_\infty)_\infty \cap \left( -(\{0\}\times\R_+)_\infty \right)
		=
		(\epi \varphi)_\infty \cap (\{0\}\times\R_-)
	\end{equation*}
	is obtained, and the latter intersection only contains the origin.
	Indeed, if $(0,-s)\in(\epi \varphi)_\infty$ for some $s>0$, then the definition of the recession cone yields
	$(z,\varphi(z)-s)\in\epi \varphi$, i.e., $\varphi(z)\leq \varphi(z)-s$ for each $z\in\dom\varphi$
	which is a contradiction.
	Thus, we can apply \cite[Thm~4]{McLindenBergstrom1981} which yields
	\begin{equation*}
		t\epimult\varphi^\gamma 
		=
		(t\epimult\varphi)\square \frac{1}{2t\gamma}\norm{\cdot}^2
		\epito
		\varphi_\infty \square \indicator_{\{0\}}
		=
		\varphi_\infty
	\end{equation*}
	as $t\downtoneq 0$.
	Observing that $0\in\dom \varphi^\gamma=\YY$ holds, \cite[Thm~8.5]{Rockafellar1970} yields
	\[
	(\varphi^\gamma)_\infty(z) = \lim_{t\downarrow 0}(t\epimult\varphi^\gamma)(z)
	\]
	for each $z\in\YY$.
	Together with the definition of epi-convergence, we have, for each $z\in\YY$,
	\begin{equation*}
		\varphi_\infty(z)
		\leq
		\liminf_{t\downtoneq 0} (t\epimult\varphi^\gamma)(z)
		= 
		\lim_{t\downtoneq 0}(t\epimult\varphi^\gamma)(z)
		=
		(\varphi^\gamma)_\infty(z)
		\leq
		\varphi_\infty(z),
	\end{equation*}
	where the last estimate follows by definition of the recession function 
	and the fact that $\varphi^\gamma$ approximates $\varphi$ from below.
	This completes the proof.
\end{proof}
		
\end{document}